\documentclass{article}

\usepackage{amsmath,amsfonts,amsthm,amssymb,amscd,color,xcolor,mathrsfs, mathtools,mathrsfs,enumitem}
\usepackage{microtype}
\usepackage{hyperref}

\usepackage{geometry}
\usepackage{makecell, booktabs}

\usepackage{comment}

\usepackage{tommasos-commands}

\usepackage{comment}

\colorlet{darkblue}{blue!50!black}
\colorlet{darkmagenta}{magenta!80!black}

\hypersetup{
    colorlinks,%
    citecolor=darkblue,% 
    filecolor=red,% 
    linkcolor=darkblue,%
    urlcolor=blue,%
    pdfnewwindow=true,%
    pdfstartview={FitH}
}

\newcommand{\diver}{\mathop{\rm div}\nolimits}

\makeatletter
\let\@fnsymbol\@arabic
\makeatother

\begin{document}
\title{Synchronisation for scalar conservation laws via Dirichlet boundary}
\author{Ana Djurdjevac\footnote{Freie Universität Berlin, Germany; e-mail: \href{mailto:adjurdjevac@zedat.fu-berlin.de
}{adjurdjevac@zedat.fu-berlin.de
}}\and Tommaso Rosati\footnote{University of Warwick, UK; e-mail:
\href{mailto:t.rosati@warwick.ac.uk}{t.rosati@warwick.ac.uk}}}
\date{\today}

\maketitle

\maketitle
\begin{abstract}
 
    We provide an elementary proof of geometric synchronisation for scalar conservation
    laws on a domain with Dirichlet boundary conditions. Unlike 
    previous results, our proof does
    not rely on a strict maximum principle, and builds instead on a
    quantitative estimate of the dissipation at the boundary. We identify a
    coercivity condition under which the estimates are uniform over all initial
    conditions, via the construction of suitable super- and sub-solutions. In lack of such coercivity our results build on $ L^{p} $
    energy estimates and a Lyapunov structure.


\end{abstract}

\setcounter{tocdepth}{1}
\tableofcontents

%\newpage

\section*{Introduction}

This article concerns long-time properties of stochastic scalar conservation laws
\begin{align}
\partial_{t} u+ \diver A(u) - \nu \Delta u   = \xi \;, \qquad
    u(0,x)=u_0(x) \;,\label{eqn:main}
\end{align}
for $ t \geqslant 0 $ and $ x \in \mD $, where $ \mD $ is a bounded domain with
homogeneous Dirichlet boundary conditions
\begin{align}
 u \big\vert_{(0, \infty) \times \partial \mD} = 0 \;. \label{e:bd}
\end{align}
In addition, $ A = (A_{i})_{i=1}^{d} \colon \RR \to \RR^{d} $ is a vector-valued nonlinearity, $
u_{0} \in C(\overline{\mD}) $ an initial value, 
$ \xi $ a noise and $ \nu > 0 $ a viscosity parameter (later on, we will fix $
\nu = 1 $). Our aim is to prove exponential synchronisation:
\begin{align}\label{e:synchro}
\| u^{1}_{t} - u^{2}_{t} \|_{L^{1}} \lesssim_{ u_{0}^{1}, u_{0}^{2}}  e^{- \mu t} \;,
\qquad \forall t \gg 1 \;,
\end{align}
for $ u^{i} $ solution to \eqref{eqn:main} with initial condition $
u^{i}_{0} $ and $ \mu > 0 $ some positive parameter.

%Scalar conservation laws such as Burgers' equation, corresponding to $
%A(u) = u^{2} $ in $ d=1 $, are toy models for fluid dynamics. 
%Questions concerning their long-time behaviour, especially in
%the small viscosity regime ($ \nu \ll 1 $), are therefore particularly interesting and 
%motivated by the study of mixing and turbulence \cite{Boritchev2014Turbulence}. 
%At the same time, Burgers' equation driven by the spatial derivative of
%space-time white noise is linked to the KPZ equation and hence to the evolution
%of growing interfaces, which has seen enormous progress
%in recent years and led to the development of singular SPDEs. Finally,
%ergodicity for Burgers' is also linked to the study of Lyapunov exponents
%for scalar linear equations \cite{DunlapGu21, GuKomorowski21Fluctu}.

Both ergodicity and synchronisation for scalar conservation laws have been
much studied in past, motivated for instance by interest in mixing and turbulence for toy
models in fluid dynamics.  First results were obtained by Sinai \cite{Sinai1991Buergers} in the periodic and viscous setting, then extended
to the inviscid regime in a celebrated work \cite{WeinanKhaninMazelSinai}.
Later works have concentrated on the analysis of the inviscid case ($ \nu \ll 1
$ or $ \nu = 0 $) or on
infinite volume. In the latter case, we cite just some of the most recent works
and refer the reader to the many references therein. Relevant to our discussion
are for example results on the existence and uniqueness of invariant measures for
Burgers' on the line \cite{Dunlap21BurgersLine} and a recent existence result
for a general class of viscous scalar conservation laws \cite{Dunlap22sclLine}. 

On finite volume much research has concentrated on the study of
synchronisation either uniformly over the viscosity $ \nu $ (in the regime $
\nu \ll 1 $) or for the inviscid case $ \nu = 0 $. We highlight in this direction a series
of results by Boritchev \cite{Boritchev18Exponential,
Boritchev2016MultidimBurgers, Boritchev2014Turbulence} (the latter being a
review article) and by Debussche and Vauvelle
\cite{DebusscheVovelle2012Existence,DebusscheVovelle2015Invariant}, see also
the many references therein. The work \cite{Boritchev18Exponential} establishes an
exponential convergence to the invariant measure in the inviscid case, building
on the particular Lagrangian structure appearing if $ \nu = 0 $. 
This result does not extend to the case $ 0 < \nu \ll 1 $. Instead, if $ \nu > 0 $ only
polynomial rate of convergence is known \cite{Boritchev2016MultidimBurgers},
uniformly over the viscosity parameter. Such result builds on Kruzhkov's
maximum principle in the periodic setting, which guarantees a bound on the
derivative of the form $ |\partial_{x} u | \leqslant C t^{-1} $, if no forcing is present
and for strictly convex $ A $. The result by Debussche and Vauvelle
\cite{DebusscheVovelle2015Invariant} proves instead uniqueness of the invariant
measure in the inviscid case beyond the convexity assumption used by Boritchev,
but their result does not provide uniform convergence rates over $ \nu $. All
mentioned results work in particular settings in which the random forcing $
\xi $ in \eqref{eqn:main} allows for ``small-noise zones'' (in the language of
\cite{Boritchev2016MultidimBurgers}), namely where the forcing vanishes, or is
particularly small, for some time: in these zones one can make use of the $
t^{-1} $ decay provided by the maximum principle to obtain synchronisation with
a polynomial rate of convergence.
Such argument does not allow, for instance, for $ \xi $
to be a deterministic, time-independent and non-zero function.

By contrast, in the viscous and periodic case exponential convergence holds for
just about any noise, as long as the equation is well defined, as a consequence
of a strong maximum principle, following for example the original argument by
Sinai. In fact, since scalar conservation laws contract over
time in $ L^{1} $, the aim in proving synchronisation is to show that the
contraction is strict most of the time, providing quantitative estimates on its
rate. For viscous equations with periodic boundary the contraction constant can
be controlled by the (strictly positive) infimum of the fundamental solution to
a linear equation.

This leads us to the present case of viscous equations with Dirichlet boundary
conditions. We immediately remark that our results do not hold uniformly over
the viscosity parameter, and indeed in the following we will assume that
$$ \nu = 1 \;. $$ 
Yet even in this setting exponential contraction is not at first evident, since
the original argument through the strict maximum principle does not work: the
infimum of a fundamental solution to a linear equation with Dirichlet boundary
is zero, attained at the boundary. A natural idea is to fix this issue by
separating the boundary effect from the bulk behavior. Such is the approach
taken by Shirikyan and coauthors (including one of the present authors)
\cite{Shirikyan18Mixing, Shirikyan18ControlTheory,
djurdjevac2022stabilisation}, yet these results require
the presence of a non-degenerate force and the solution of a control problem.

In the present work we take a different approach, and provide a bound
on the contraction constant through an elementary estimate on the dissipation
of mass at the boundary, which does not require any non-degeneracy condition.
In particular, our approach solves for example the second open problem in
\cite{djurdjevac2022stabilisation}.
In our proof, such dissipation at the boundary is an effect of the viscosity, rather
than the nonlinearity, so there is no hope for
estimates that are uniform over $ \nu $. Yet we believe that if
the forcing additionally satisfies conditions leading to  ``small-noise regions'' as
described above, then  under the assumption that $ A
$ is coercive (a weaker assumption than the uniform convexity in the work by
Boritchev) similar uniform estimates as those presented in
\cite{Boritchev2016MultidimBurgers} should hold. 

Coercivity appears
because under such assumption in presence of Dirichlet boundary one can construct explicit super- and
sub-solutions which guarantee (in absence of forcing) that $ | u_{t} | \leqslant
Ct^{-1} $ for some $ C >0 $ depending on the domain and uniform over all
initial conditions. These bounds take the r\^ole of
Kruzhkov's maximum principle in \cite{Boritchev2016MultidimBurgers}. They were
employed already in \cite{Shirikyan18Mixing, djurdjevac2022stabilisation}, and
are related to $ N $-waves
and the appearance of shocks in scalar conservation laws. In lack of coercivity
we still obtain our result, but without uniformity over all initial conditions,
and we rely on $ L^{p} $ energy estimates in the spirit of
\cite{Gyiongy00Wellposed}.

The construction of explicit super- and sub-solutions is particularly
interesting as it suggests a possibility of proving coming down from infinity,
a property understood for $ \Phi^{4}_{d} $ equations
\cite{MourratWeber17Infinity}, and global in time well-posedness for a wide
class of scalar conservation laws driven by irregular noise (cf.\
\cite{ZhuZhu2022HJB, PerkowskiRosati2019KPZ} for results for KPZ on the real line)
without the use of the Cole-Hopf transformation, which is in general not
available. 

We leave this question to future investigations. Similarly, we did not
present here the most general possible setting: we believe that our results 
apply for instance to inhomogeneous boundary or quasilinear strictly
elliptic equations. In addition, synchronisation as in \eqref{e:synchro} is only
one exemplary longtime property, and following similar arguments as in our proofs one can
recover the existence a one-point attractor, often referred to as a one force,
one solution principle.

%\textcolor{red}{Shirykian - maybe the work from 2017 on control of deterministic Burgers'}

\subsection*{Acknowledgments}
The authors are very grateful to Armen Shirikyan for several interesting
discussions and comments. The research of ADj has been partially funded 
by Deutsche Forschungsgemeinschaft (DFG) through the grant CRC
1114 ``Scaling Cascades in Complex Systems", Project Number 235221301. 
This work was written in vast majority while TR was employed at Imperial
College London. TR is very grateful to Martin Hairer and the Royal Society for financial support
through the research professorship grant RP\textbackslash R1\textbackslash
191065.

\subsection*{Notation and conventions}

We define $ \NN = \{ 0, 1, 2,\dots \} $ and $ \RR_{+} = [0, \infty) $. For any
%\tommaso{We should always assume that $ 0 \in \mD $, otherwise our bound on the
%super-solution does not work the same way we have written it.}
%\ana{but we could modify it?}
bounded measurable set $ \mD \subseteq \RR^{d} $ and $ p \in [1, \infty] $
consider $ L^{p} (\mD) $ the space of functions (modulo changes on null-sets)
defined by the norm $ \| \varphi \|_{L^{p}} = \left( \int_{\mD} | \varphi
(x) |^{p} \ud x \right)^{\frac{1}{p}} $, with the usual convention for $
p = \infty $. 
For brevity we shall use the shorthand notation $ \| \varphi \|_{p} = \| \varphi
\|_{L^{p}} $. We write $ \langle f, g \rangle = \smallint_{\mD} f(x)
g(x) \ud x $ whenever the integral is defined. For any two topological spaces $
X, Y $ we write $ C(X; Y) $ for
the space of continuous functions $ f \colon X \to Y $, with the topology of
uniform convergence on all compact sets. If $ Y = \RR $ we simply write $
C(X) = C(X; \RR) $. Similarly, for $ m, n, k \in \NN $ we write $
C^{k}(\RR^{m}; \RR^{n}) \subseteq C(\RR^{m}; \RR^{n}) $ for the space of $ k $ times differentiable
functions, with all derivatives continuous (and the topology of uniform
convergence on compact sets for all derivatives). Function $F = (F_{i})_{i =1}^{d}$ belongs to  $C^{1}(\mD; \RR^{d}) $ if and only if for every $i=1, \dots, d$, $D^\alpha F_i$ continuously extends to $\overline{\mD}$ for every multiindex $\alpha$, $|\alpha | \leq k$. %\ana{check $
%C^{k} (\mD; \RR^{m}) $}   
For $ F = (F_{i})_{i =1}^{d} \in C^{1}(\mD; \RR^{d}) $ we define $ \div(F) = \sum_{i
=1}^{d} \partial_{x_{i}} F_{i} $.
For any set $ \mX $ and functions $ f, g \colon
\mX \to \RR $ we write $$ f \lesssim g $$ if there exists a constant $ C >0 $
such that $ f(x) \leqslant C g(x)$ for all $ x \in \mX $. If the
proportionality constant depends on some parameter $ \vt $, we may highlight
this by writing $ f \lesssim_{\vt} g $.
For any $\gamma \in (0,1]$ we denote by $C^\gamma(\overline{\mD})$ the H\"older space. 
We will denote by $ P_{t} $ the Dirichlet heat semigroup on $ \mD $, meaning that  $ P_{t} f $ solves $ \partial_{t}
        (P_{t} f) =  \Delta (P_{t} f) $ on $ \mD $ with $ P_{0}f = f $ and
    $ P_{t} f = 0 $ on $ \partial \mD $. 

\section{Main results}

 The precise setting in which we will work is described below.

\begin{assumption}\label{assu:noise}
Here we collect the requirements on the domain, the boundary conditions, the
nonlinear vector field $ A $ and the noise.
    \begin{enumerate}
        \item \textbf{(Domain and boundary conditions)} Let $ \mD \subseteq \RR^{d} $ be and
            open domain with $ C^{2} $ boundary $ \partial \mD $.

        \item \textbf{(Nonlinearity)} We assume that $ A \in
            C^{1}(\RR ; \RR^{d}) $, and that there exists an $ \mf{ a} \in [1, \infty]
            $ such that one of the following conditions holds, depending on the
            value of $ \mf{a} $.
            \begin{itemize}
                \item \textbf{(Component-wise coercivity: if $ \mf{a} = \infty
                    $)} There exist constants $\alpha, \beta \in \RR_{+} $ such
                    that
                    \begin{equation}\label{e:coerc}
                    \begin{aligned}
                        A_{i}^{\prime}&( u) \mathrm{sign}(u)  \geqslant \alpha | u | -
                        \beta \;, \qquad \forall i \in \{ 1, \dots, d \} \;,  u \in \RR \;.
                    \end{aligned}
                  \end{equation} 
              \item \textbf{(Polynomial growth: if $ \mf{a} < \infty $)} There
                  exists a constant $C >0 $ such that
                  \begin{align*}
                      | A(u) | \leqslant C(1 + | u |)^{\mf{a}} \;, \qquad |
A^{\prime} (u) | \leqslant C (1 + | u |)^{\mf{a}} \;.
                  \end{align*} 
            \end{itemize}
        \item \textbf{(Noise)} Consider a probability space $ (\Omega, \mF,
            \PP) $ enhanced with a map $ \vt \colon \ZZ \times \Omega \to
            \Omega $ such that $ (\Omega, \mF, \PP, \vt) $ is an
            \emph{ergodic, invertible metric dynamical system}, supporting for
            some $ n \in \NN $ cocycles $$ \psi_{0} \colon \Omega \to C(\RR_+ \times
            \overline{\mD}) \;, \quad \text{ and } \quad  B^{k} \colon \Omega \to
            C(\RR_+) \;, \quad \forall k \in \{ 1, \dots, n \} \;,
            $$
            such that $ \{ B^{k}
            \}_{k = 1}^{n} $ is a set of i.i.d.\ Brownian motions and
            $ \psi_{0} $ satisfies for any $ p \geqslant 0 $ the moment
            estimate
            \begin{equation}\label{e:psi-0-mmt}
                \sup_{t \geqslant 0} \EE \| \psi_{0} (t , \cdot)
                \|_{\infty}^{p} < \infty \;.
            \end{equation}
            Then define
            \begin{align*}
                \xi(\omega, t, x) = \psi_{0} (\omega, t, x) + \sum_{k
                =1}^{n} \psi_{k}(x) \dot{B}^{k}_{t}(\omega) \;,
            \end{align*}
            for some $ \{ \psi_{k} \}_{k =1}^{n} \subseteq
            C^{\alpha_{0}}( \overline{\mD}) $, for some $ \alpha_{0} \in (0, 1) $. 
        If $ \mf{a} < \infty $ (for $ \mf{a} $ as in the previous point of the
        assumption), then we additionally assume that $ \psi_{0} $ is not
        random, that is $ \psi(\omega, t,x) = \psi_{0}(t, x)$, so that
        \eqref{e:psi-0-mmt} reads
        \begin{align*}
            \sup_{t \geqslant 0} \sup_{x \in \mD} | \psi_{0} (t, x) | <
            \infty \;.
        \end{align*}
        \end{enumerate}
\end{assumption}

We observe that under Assumption~\ref{assu:noise} the noise $ \xi $ can
potentially be degenerate, indeed even $
\xi (\omega, t, x) = h (x) $ for a H\"older-continuous, time-independent function $ h
$ is allowed. This reflects the fact that the main mechanism behind
synchronisation is the structure of the equation, and in
particular order preservation, cf.\ Lemma~\ref{lem:l1-contraction}, which leads
to an $ L^{1} $ contraction. In particular, our choice of $ \xi $ covers the
setting of \cite{djurdjevac2022stabilisation} and \cite{Shirikyan18Mixing},
where the authors require respectively a one-dimension and a two-dimensional
noise to obtain mixing. We remark the slight difference between the settings $
\mf{a} = \infty $ and $ \mf{a} < \infty $. In the first we have extremely
strong $ L^{\infty} $ estimates, which allow us to treat also non-Markovian
settings (which is the case if $ \psi_{0} $ is allowed to be a generic
cocycle): instead, if $ \mf{a} < \infty $, we need to use $ L^{p} $ bounds that
are less strong (in particular not uniform over all initial conditions). It is
then convenient to have some kind of Markov structure on the solutions $ u
$: the requirement that $ \psi_{0} $ be deterministic guarantees such a
structure, but milder assumptions might be sufficient (for example requiring a
finite range of dependence).

As for the regularity of the noise, it is of course of extreme interest to
understand whether our methods apply to low regularity, for example due to the
connection with the Burgers' equation driven by the gradient of space-time white
noise, which is linked to the KPZ equation \cite{Hairer13Solving}. Our
assumption that $ \psi_{k} \in C^{\alpha}  $ for $ k \in \{ 1, \dots, n \} $
and $ \alpha > 0 $ reflects the fact that we need $ z = (\partial_{t} -
\Delta)^{-1} \xi \in C^{\beta} $ for some $ \beta \geqslant 1 $, cf.\
Lemma~\ref{lem:reg-z}: we need $
z $ differentiable to be able to construct the super-solution $
\varphi^{+}  $ to \eqref{eqn:for-w} that is required for the $ L^{\infty} $
a-priori estimate in the case $ \mf{a} = \infty $, see the proof or
Proposition~\ref{prop:wp}. 
%This requirement is mostly
%technical, and in a forthcoming work we will
%show how to construct similar super- and sub-solutions for significantly more irregular noises.

As for the nonlinearity, the two requirements in Assumption~\ref{assu:noise}
are substantially different. If $ \mf{a} = \infty $ we do not need any
growth assumption on $ A $ , as we can use the coercivity to
construct super- and sub-solutions that allow to control the norm $ \| u_{t}
\|_{\infty} $ at any positive time $ t > 0 $, uniformly over all initial
conditions $ u_{0} $: this control is somewhat in the spirit of the coming down from
infinity property of the $ \Phi^{4}_{d} $ model \cite{MourratWeber17Infinity}. 
The coercivity assumption we impose holds \emph{component-wise}, namely for
each component $ A^{\prime}_{i} $ for $ i \in \{ 1, \dots, d \} $. In the case
of smooth (in both space and time) driving noise this condition can be relaxed
to the more natural
\begin{align*}
    \sum_{i =1}^{d} A^{\prime}_{i} (u) \mathrm{sign} (u) \geqslant \alpha | u | - \beta \;.
\end{align*}
On the other hand, in the case $ \mf{ a} < \infty $ we do not assume any particular structure
on $ A $, other than some polynomial growth. Here we use $ L^{p} $ energy
estimates in the spirit of \cite{Gyiongy00Wellposed}, and our results are not
uniform over the initial condition. 

Finally, the requirement that the boundary $ \partial \mD $ is of class $
C^{2} $ is purely technical, so that we can apply the theory of analytic
semigroups as presented in \cite[Chapters 2, 3]{Lunardi95Analytic}. 

In the setting we have introduced, let us now clarify the notion of 
solution that we will work with.  It will be useful, also for later
convenience, to decompose the solution $ u $ to \eqref{eqn:main} as follows:
\begin{align*}
    u =   z + \varphi \;,
\end{align*}
where $ z $ contains the stochastic forcing and $ \varphi $ all the rest. Namely
%\begin{equation}\label{e:zeta}
%    \partial_{t} \zeta^{g} = \nu \Delta \zeta^{g} \;, \qquad \zeta^{g} (0, \cdot) = 0 \;, \qquad
    %\zeta^{g} \big\vert_{(0, \infty) \times\partial \mD}  = g \;,
%\end{equation}
%and
    \begin{equation}\label{e:z} 
        \partial_{t} z =  \Delta z + \psi_{0} + \sum_{k =1}^{n} \psi_{k}
        \dot{B}^{k}_{t} \;, \qquad z(0, \cdot ) = 0 \;, \qquad z
        \big\vert_{(0, \infty) \times \partial \mD}(t, \cdot) = 0 \;,
    \end{equation}
so that $\varphi $ should solve
    \begin{equation}\label{eqn:for-w}
        \partial_{t} \varphi =  \Delta \varphi - \div (A (\varphi + z)) \;, \qquad
        \varphi (0, \cdot) = u_{0}(\cdot) \;, \qquad \varphi
        \big\vert_{(0, \infty) \times \partial \mD} = 0  \;.
    \end{equation}

\begin{definition}\label{def:weak-solutions}
%    \tommaso{There is no boundary condition in this definition\dots I also
    %wonder if written in weak for uniqueness is simple to check..}
    Under Assumption~\ref{assu:noise}, for any $ \varrho \in  [\mf{a}, \infty]
    $ and $ u_{0} \in L^{\varrho} $ we say that $ u $ is a mild
    solution to \eqref{eqn:main} if $ u =  z + \varphi $, with $ z $ satisfying \eqref{e:z} and $
    \varphi \in C([0, \infty); L^{\varrho} ) $ given by
    \begin{align*}
        \varphi_{t}  = P_{t} u_{0} - \int_{0}^{t} P_{t -s} \div (A
        (\varphi_{s} + z_{s} )) \ud s \;.
\end{align*}

\end{definition}

The next result guarantees that Equation~\eqref{eqn:main} is well-posed for all
times, together with suitable a-priori estimates, depending on the choice of $
\mf{a} $ in Assumption~\ref{assu:noise}.

\begin{proposition}\label{prop:wp}
    Under Assumption~\ref{assu:noise} there exists a $ \varrho (\mf{a}, d) \in
    [\mf{a}, \infty] $ such that for
    any $ u_{0} \in L^{\varrho} $ there exists a unique mild solution to
    \eqref{eqn:main} in the sense of Definition~\ref{def:weak-solutions}. In addition
    for some $ C_{1}, C_{2} \in (0, \infty) $
    \begin{align*}
        & \sup_{t \geqslant 1}  \EE  \left[  \sup_{u_{0} \in L^{\infty}}\sup_{s \in [t,
        t+1]}\| u_{s} \|_{\infty}\right]   \leqslant C_{1} \;, \qquad & & \text{
        if } \mf{a} = \infty  \;, \\
        & \sup_{t \geqslant 1} \EE  \left[ \sup_{s \in [t,
        t+1]}\| u_{s} \|_{ \infty}\right]  \leqslant C_{1} (
            e^{- C_{2} t}\| u_{0} \|_{L^{\varrho}} + 1) \;, \qquad & & \text{
        if } 1 \leqslant \mf{a} < \infty  \;.
    \end{align*}
\end{proposition}
The main result of this work is then the following theorem.
\begin{theorem}\label{thm:synchro}
    For any $ u_{0}, v_{0} \in L^{\varrho}$, with $ \varrho (\mf{a},d) $ as in
    Proposition~\ref{prop:wp}, denote with $ u_{t} , v_t $ the 
    solution to \eqref{eqn:main} with initial datum $ u_{0},
    v_{0} $ respectively.
    Then there exists a constant $ C > 0 $ such that
    \begin{align*}
        \lim_{t \to \infty} \sup_{u_{0}^{1}, u_{0}^{2} \in
        L^{\infty}_{0}} \frac{1}{t} \log{ \left( \| u_{t} - v_{t}
        \|_{L^{1}} \right) } & \leqslant - C \;, \qquad & & \text{
        if } \mf{a} = \infty \;, \\
        \sup_{u_{0}^{1}, u_{0}^{2} \in
        L^{\varrho}_{0}} \lim_{t \to \infty}  \frac{1}{t} \log{ \left( \| u_{t} - v_{t}
        \|_{L^{1}} \right) } & \leqslant - C \;, \qquad & & \text{
        if } 1 \leqslant \mf{a} < \infty \;. 
    \end{align*}
\end{theorem}

Although the statement of the theorem considers only synchronization in the $
L^{1} $ norm, which is the simplest possible choice, our result can be lifted to
higher regularity norms following similar arguments as in
\cite{Rosati22Synchro} or \cite{Rosati21Lyap}.

\section{A priori estimates}

This section is essentially devoted to a proof of Proposition~\ref{prop:wp},
plus a  corollary.
    Eventually we will require an estimate for $ \varphi $, but let us start
    with estimates on $ z $. Here we require parabolic regularity estimates for the
    Dirichlet heat semigroup $
    P_{t} $ on $ \mD $.
    Namely, we will use that  $ P_{t} $ is an analytic semigroup, cf.
    \cite[Chapters 2, 3]{Lunardi95Analytic}, so that for any $ 0 < \alpha \leqslant  \beta
    < \infty$ with $ \alpha, \beta \not\in \NN $ one can estimate, for some $ \lambda > 0$
    \begin{equation}\label{e:schauder}
        \sup_{t \geqslant 0} \left\{ t^{\frac{\beta - \alpha}{2}} e^{\lambda t} \| P_{t} f
        \|_{C^{\beta }( \overline{\mD})} \right\} \lesssim_{\alpha, \beta}  \|
            f \|_{\alpha} \;.
        \end{equation}
        This estimate leads us to the following uniform bound on $ z $.

    \begin{lemma}\label{lem:reg-z}
        Under Assumption~\ref{assu:noise}, the process $ z $ defined by
        \eqref{e:z} satisfies for some $ \gamma > 1, p > 0 $ and $ T > 0 $ the
        following moment bound 
    \begin{equation}\label{e:mmt-z}
        \sup_{t \geqslant 0}\EE \left[ \sup_{t \leqslant s \leqslant t+ T} \| z
        \|_{ C^{\gamma} ( \overline{\mD})}^{p} \right] \leqslant C(T, \gamma,
        p) \;,
    \end{equation}
    for some $ C(T, \gamma, p) \in (0, \infty) $.
    \end{lemma}

    \begin{proof}
        Let us rewrite (omitting the dependence on $ \omega $)
        \begin{align*}
            z (t, x) = \int_{0}^{t}[ P_{t -s} \psi_{0}(s, \cdot)] (x) \ud s + \sum_{k
            =1}^{n} \int_{0}^{t} [P_{t-s} \psi_{k}] (x) \ud B^{k}_{s}.
        \end{align*}
        Since the first term is simpler than the latter terms, let us just treat
        one of the last addends.  Define $ z^{k} (x) := \int_{0}^{t} [P_{t -s}
        \psi_{k}] (x) \ud s $ and for $ \alpha_0 $ as in
        Assumption~\ref{assu:noise},  fix any $ 1 < \gamma < 1 +
        \alpha_{0} $. Then
        for $ p \geqslant 2$ by BDG we have for any $ x, y \in \mathcal{D} $
        \begin{align*}
            \EE \left[ | \nabla z^{k}_{t} (x) - \nabla z^{k}_{t} (y) |^{p} \right] & \lesssim
            \bigg( \int_{0}^{t} | \nabla P_{t -s} \psi_{k} (x) -  \nabla P_{t - s} \psi_{k} (y)
            |^{2}  \ud s \bigg)^{\frac{p}{2}} \\
            & \lesssim | x - y |^{p(\gamma-1)} \int_{0}^{t} \| P_{t - s}
            z^{k} \|^{2}_{C^{\gamma}( \overline{\mD})} \ud s \\
            & \lesssim | x - y |^{p (\gamma - 1)} \int_{0}^{t} e^{- 2 \lambda(t -s)} t
            ^{- \frac{\gamma - \alpha_0}{ 2}} \ud s \lesssim | x - y |^{p
            (\gamma -1)}\;,
        \end{align*}
        so that by \eqref{e:schauder}, by our continuity assumption on $ \psi_{k}
        $ in Assumption~\ref{assu:noise}, and by the Kolmogorov continuity test we
        obtain for any $ \gamma^{\prime}  < \gamma $
        \begin{align*}
            \sup_{t \geqslant 0}\EE \| z (t, \cdot) \|_{C^{\gamma^{\prime} } (\mD)}^{p} < \infty \;.
        \end{align*}
        From here \eqref{e:mmt-z} follows along similar lines to
        those we have just sketched to obtain some temporal regularity, which
        eventually guarantees the uniform estimate in time.
    \end{proof}

    Now we are ready to prove the uniform bounds in Proposition~\ref{prop:wp}.

\begin{proof}[Proof of Proposition~\ref{prop:wp}]

    \textit{Case $ \mf{a} = \infty $.} As for the well-posedness of the equation for initial data  $ u_{0} \in L^{\infty}
    $, we will consider mild solutions in the sense of the Definition \ref{def:weak-solutions}.
    %, denoting with $ P_{t} $ the Dirichlet
   % heat semigroup on $ \mD $, which satisfies the Schauder estimates in
    %\eqref{e:schauder}. 
    %In the setting of Assumption~\ref{assu:noise} with $
  %  \mf{a} = \infty $, we say that $ u $ is a solution to \eqref{eqn:main} on
   % $ [0, T] $, for some $ T > 0 $ and with initial condition $ u_{0} $ if $ u $ is
    %of the form $ u = z + \varphi $,
    %with $ z $ solving \eqref{e:z} and $ \varphi \in C((0, T] \times
    %\overline{\mD} ) $ being a mild solution to
    %\eqref{eqn:for-w}, meaning that 
    %\begin{align*}
     %   \varphi_{t} = P_{t} u_{0} + \int_{0}^{t} P_{t-s} [ \div (A
      %  (\varphi_{s} + z_{s}))]\ud s \;, \qquad \forall t \in [0, T] \;.
    %\end{align*}
    Here through a classical Picard contraction argument one obtains that for
    any $ M > 0 $ and almost all $ \omega \in \Omega $ there exists a $
    T^{\mathrm{fin}}(M, \omega) > 0 $ such that
    for all $u_{0} \in L^{\infty}$ satisfying $\| u_{0} \|_{\infty} \leqslant
    M$, there exists a unique mild solution  $\varphi$ to
    \eqref{eqn:for-w} on $[0, T^{\mathrm{fin}}(M, \omega))$.

    To construct global solutions in time we need an a-priori $ L^{\infty} $ estimate.
   % Let us assume that we are given a
    %smooth local solution $ \varphi $ to \eqref{eqn:for-w}, so $ \varphi \in C^{\infty} ([0, T] \times
    %\mD) $ for some (random) $ T > 0 $.
  %  Assume first that $ u_{0} \in C^{\infty}( \overline{\mD}) $.  We want to establish an a-priori $
    %L^{\infty} $ estimate. 
    For this reason define
    \begin{align*}
        \varphi^{+}(t, x) = \frac{a + b\sum_{i =1}^{d} x_{i}}{ t}\;.
    \end{align*}
    Let us fix an arbitrary $ T > 0 $. We will then find $ a(T, \omega), b(T,
    \omega) > 0 $ such that $ \varphi^{+} $ is a
    super-solution to \eqref{eqn:for-w} on $ (0, T] $ for any initial
    condition  $ u_{0} \in L^{\infty} $, in the sense that if $ \varphi $ is
    the solution to \eqref{eqn:for-w} with initial condition $ u_{0} $, then
    \begin{align*}
        \varphi (t, x) \leqslant \varphi^{+}(t, x) \;, \qquad \forall t \in
        (0, T^{\mathrm{fin}}( \| u_{0} \|_{\infty}, \omega) \wedge T)\;, x \in \mD \;.
    \end{align*}
    To prove that $ \varphi^{+} $ is a super-solution we have to fix the
    parameters $ a , b $. We start by defining
    \begin{align*}
        a = b( 1 + D + T \| z \|_{\infty, T} + T \beta / \alpha ) \;,
    \end{align*}
    where $ \alpha, \beta $ are as in the coercivity requirement of Assumption~\ref{assu:noise} and where
    \begin{align*}
        D = \max_{x \in \overline{\mD}} \sum_{i=1}^{d}| x_{i} | \;, \qquad \|
        z \|_{\infty, T} = \sup_{0 \leqslant t \leqslant T} \| z_{t} \|_{\infty}
        \;.
    \end{align*}
    Such choice of $ a $ will be motivated by the upcoming calculations.
    Indeed, we find that
    \begin{align*}
        \partial_{t} \varphi^{+} + \mathrm{div} (A(\varphi^{+}+ z)) - \Delta \varphi^{+} & = -
        \frac{\varphi^{+}}{t} + \sum_{i=1}^{d}  A_{i}^{\prime}  (\varphi^{+} + z)
        \partial_{i} ( \varphi^{+} + z) \\
        & = -
        \frac{\varphi^{+}}{t} + \sum_{i=1}^{d}  A_{i}^{\prime}  (\varphi^{+} + z) \left[
        \frac{b}{t} + \partial_{i} z\right] \;.
    \end{align*}
    Our aim is then to show that if $ b > 1 $ is chosen sufficiently large, then
    \begin{equation}\label{eqn:aim}
         \sum_{i=1}^{d}  A_{i}^{\prime}  (\varphi^{+} + z) \left[
         \frac{b}{t} + \partial_{i} z\right] \geqslant \frac{\varphi^{+}}{t} \;.
    \end{equation}
    Let us observe that by construction, since $ b \geqslant 1 $
    \begin{equation}\label{eqn:bds-w-+}
        \frac{\beta}{\alpha} + b t^{-1} - z \leqslant \varphi^{+} \leqslant b
        t^{-1}(1 + 2 D + T \| z \|_{\infty, T}+ T \beta / \alpha) \;.
    \end{equation}
    In particular it suffices to prove
    \begin{align*}
        \sum_{i=1}^{d}  A_{i}^{\prime}  (\varphi^{+} + z) \left[
        \frac{b}{t} + \partial_{i} z\right] \geqslant \frac{1 + 2 D + T\| z
        \|_{\infty, T}+ T \beta / \alpha}{b}  \left(
        \frac{b}{t}\right)^{2} \;.
    \end{align*}
    Using $ \varphi^{+} + z \geqslant \beta / \alpha $, which in turn, by our
        component-wise coercivity requirement in Assumption~\ref{assu:noise}, implies $
        A^{\prime}_{i} (\varphi^{+} + z) \geqslant 0 $, we can bound
    \begin{align*}
        \sum_{i=1}^{d}  A_{i}^{\prime}  (\varphi^{+} + z) \left[
        \frac{b}{t} + \partial_{i} z\right] & \geqslant \sum_{i=1}^{d}  A_{i}^{\prime}  (\varphi^{+} + z) \left[
        \frac{b}{t} - \| \partial_{i} z \|_{\infty, T}\right] \geqslant \frac{1}{2} \frac{b}{t} \sum_{i=1}^{d}  A_{i}^{\prime}
        (\varphi^{+} + z) \;,
    \end{align*}
    if we assume that 
    \begin{equation}\label{eqn:b-1}
        b t^{-1} \geqslant 2 \sum_{i=1}^{d} \| \partial_{i} z
        \|_{\infty, T} \;.
    \end{equation}
    Note that the lower bound in \eqref{eqn:b-1} is finite almost surely, since by
    Lemma~\ref{lem:reg-z} we have $ z \in C^{\gamma} $ for some $ \gamma > 1 $. 
    In particular, under this condition and using the coercivity assumption on
    $ A $ in Assumption~\ref{assu:noise} we can
    further reduce the problem to proving
    \begin{align*}
        \frac{1}{2} \frac{b}{t}\sum_{i=1}^{d}\left( \alpha (\varphi^{+} +z)  - \beta 
        \right) \geqslant \frac{1 + 2 D + T\| z \|_{\infty, T} + T \beta /
        \alpha}{b}  \left( \frac{b}{t}\right)^{2} \;,
    \end{align*}
    and again via \eqref{eqn:bds-w-+} we reduce this to
    \begin{align*}
        \frac{1}{2} \frac{b}{t}\sum_{i=1}^{d} \alpha \frac{b}{t}  \geqslant
        \frac{1 + 2 D + T\| z \|_{\infty, T}+ T \beta / \alpha}{b}  \left(
        \frac{b}{t}\right)^{2} \;,
    \end{align*}
    which is satisfied for
    \begin{equation}\label{eqn:b-2}
        b \geqslant 2 \frac{1 + 2 D + T\| z \|_{\infty, T}+ T \beta / \alpha}{\alpha d} \;.
    \end{equation}
    Combining \eqref{eqn:b-1} and \eqref{eqn:b-2}, let us define
    \begin{align*}
        b := \max \left\{  2 \frac{1 + 2 D + T\| z \|_{\infty, T}+T \beta/
            \alpha}{\alpha d} \;,
        2 \sum_{i=1}^{d} \| \partial_{i} z \|_{\infty, T} \right\} \;.
    \end{align*}
    Then we have proven that $ \varphi^{+} $ is a super-solution. Analogously one can
    construct sub-solutions.

    \textit{Case $ \mf{a} \in [1, \infty) $.} In this case we refrain from
    proving well-posedness of the equation, as this is already
    well understood, see for example \cite[Theorem
    2.1]{Gyiongy00Wellposed}. Instead we concentrate on the uniform bounds and start by
    establishing an $ L^{p} $ energy estimate. We can compute, assuming for simplicity $ p \in 2 \NN \setminus \{
    0 \}$:
\begin{equation}\label{eqn:Lp-estimate}
    \begin{aligned}
        \ud \| u \|_{L^{p}}^{p} = & p \langle u^{p -1}, \Delta u \ud t -
        \mathrm{div}(A (u)) \ud t + \psi_{0} \ud t + \sum_{k = 1}^{n} \psi_{k}
        \ud B_{t}^{k} \rangle \\
        & + \sum_{k =1}^{n}\frac{p(p-1)}{2} \langle u^{p-2},  \psi_{k}^{2} \rangle \ud t \;.
    \end{aligned}
\end{equation}
    In the spirit of \cite{Gyiongy00Wellposed}, the core of the estimate lies in the cancellation
    \begin{equation}\label{e:cancellation}
        \langle u^{p-1}, \mathrm{div}(A(u))\rangle = 0 \;.
    \end{equation}
    In fact, one
    can rewrite by integration by parts
    \begin{align*}
        \langle u^{p-1}, \mathrm{div} (A(u)) \rangle &  = -\sum_{i = 1}^{d} \int_{\mD}
        \partial_{x_{i}} H_{i}(u) (x) \ud x + c\;, \\
        c & = \int_{\partial \mD} \mathbf{n} (x) \cdot (u^{p-1} A
        (u))(x) \ud \Sigma (x) = 0 \;,
    \end{align*}
    where $ \mathbf{n} (x) $ is the outer unit normal to the boundary $
    \partial \mD $ at $ x \in \partial \mD $, $ \Sigma $ is the $ (d-1)
    $--dimensional Hausdorff measure on the boundary and $
    H_{i} \colon \RR \to \RR $ is defined as the primitive
    \begin{align*}
        H_{i} (a) = \int_{0}^{a}  (p - 1) r^{p - 2} A_{i} (r) \ud r \;.
    \end{align*}
    Then, by the divergence theorem
    \begin{align*}
        \sum_{i = 1}^{d} \int_{\mD} \partial_{x_{i}} H_{i}(u) (x) \ud x =
        \int_{\partial \mD} \mathbf{n} (x) \cdot H (u) (x) \ud \Sigma (x)
        =0 \;.
    \end{align*}
    Therefore \eqref{e:cancellation} is proven.
    Moreover, we have that by the Poincar\'e's inequality
    \begin{align*}
        \frac{1}{p-1} \langle  u^{p -1}, \Delta u\rangle & = - \int_\mD (u)^{p-2} | \nabla u |^{2}
        \ud x = -\| u^{p/2-1}\nabla u \|_{L^{2}}^{2}  = - \frac{2}{p} \| \nabla
        (u^{\frac{p}{2}}) \|_{L^{2}}^{2} \\
        & \lesssim_{p} -\| u ^{\frac{p}{2}} \|_{L^{2}}^{2} = - \| u \|_{L^{p}}^{p} \;.
    \end{align*}
    Finally, for any $ \ve \in (0, 1) $ we can bound by Young's inequality for
    products
    \begin{align*}
        \langle u^{p - 1}, \psi_{0}  \rangle & \leqslant \| u \|_{L^{p
        -1}}^{p -1} \| \psi_{0} \|_{\infty} \leqslant \frac{p-1}{p} \ve^{\frac{p}{p-1}} \| u
        \|_{L^{p-1}}^{p} + \frac{1}{p} \ve^{- p} \| \psi_{0} \|_{\infty}^{p} \\
        \langle u^{p -2}, \sum_{k =1}^{n} \psi_{k}^{2} \rangle & \leqslant \|
        u \|_{L^{p-2}}^{p -2} \left\| \sum_{k=1}^{n} \psi_{k}^{2} \right\|_{\infty}
        \leqslant \frac{p -2}{p} \ve^{\frac{p}{p-2}} \| u \|_{L^{p
        -2}}^{p} + \frac{2}{p} \ve^{- \frac{p}{2}} \left\| \sum_{k =1}^{n} \psi_{k}^{2}
        \right\|_{\infty}^{\frac{p}{2}} \;.
    \end{align*}
    In this way, denoting with $ M^{(p)}_{t} $ the martingale
    \begin{equation}\label{eqn:martingale}
        \begin{aligned}
            M_{t}^{(p)} = \sum_{k =1}^{n} \int_{0}^{t} p\langle u^{p-1}_{s} ,
            \psi_{k}  \rangle \ud B_{t}^{k} \;,
        \end{aligned}
    \end{equation}
    we have obtained the overall estimate by choosing $ \ve $ sufficiently
    small and up to decreasing the value of the constant $ c > 0 $:
    \begin{equation}\label{eqn:lp-final}
        \begin{aligned}
            \ud \| u \|_{L^{p}}^{p} \lesssim_{p, \psi} \left\{  - c \| u
            \|_{L^{p}}^{p} + 1 + \| \psi_{0}(t , \cdot) \|_{\infty}^{p} \right\} \ud t + \ud
        M^{(p)}_{t} \;.
        \end{aligned}
    \end{equation}
    Moreover, by the BDG inequality, we can control the martingale term $
    M_{t}^{(p)} $ as follows, for any $ t , h \geqslant 0 $:
    \begin{align*}
        \EE \left[ \sup_{t \leqslant s \leqslant t+h} \left(M_{s}^{(p)} -
        M_{t}^{(p)} \right) \right] & \lesssim \EE
        \left[ \left( \int_{t}^{t+ h} \ud \langle M^{(p)} \rangle_{s}
        \right)^{\frac{1}{2}} \right] \\
        & = \EE \left[
        \left( \sum_{k =1}^{n}\int_{t}^{t+h} p^{2} \langle u^{p-1}_{s},
        \psi_{k} \rangle^{2} \ud s \right)^{\frac{1}{2}} \right]\\
        & \lesssim_{p, \psi} \sqrt{h} \EE \left[ \sup_{t \leqslant s
        \leqslant t+h} \| u_{s} \|_{L^{p-1}}^{p-1}\right] \;.
    \end{align*}
    Now we are ready to close our estimates. From \eqref{eqn:lp-final},
    together with the moment bound on $ \| \psi_{0}(t, \cdot) \| $ in
    Assumption~\ref{assu:noise}, we obtain for some $ c_{1}, c_{2} > 0 $
    \begin{equation}\label{eqn:a-priori-lp}
        \begin{aligned}
            \EE \| u_{t} \|_{L^{p}}^{p} \leqslant e^{ -c_{1} t} \EE \|
            u_{0} \|_{L^{p}}^{p} + c_{2} \;.
        \end{aligned}
    \end{equation}
    Furthermore, again from \eqref{eqn:lp-final}, we have for any $ h \in (0, 1) $ 
    \begin{align*}
        \EE \left[ \sup_{t \leqslant s \leqslant t+h} \| u_{s}
        \|_{L^{p}}^{p} \right] & \leqslant \EE [ \| u_{t} \|_{L^{p}}^{p}] +
        C_{1} \EE \left[ \int_{t}^{t+h}(1 + \| \psi_{0}(s, \cdot) \|^{p}_{\infty}) \ud s
        \right] + C_{2} \sqrt{h} \EE \left[ \sup_{t \leqslant s
        \leqslant t+h} \| u_{s} \|_{L^{p-1}}^{p-1}\right] \\
        & \lesssim \EE [ \| u_{t} \|_{L^{p}}^{p}] + \sqrt{h} \left( 1 +\EE \left[ \sup_{t \leqslant s
        \leqslant t+h} \| u_{s} \|_{L^{p-1}}^{p-1}\right] \right)  \;,
    \end{align*}
    so that by a Gronwall-type argument, choosing $ h $ sufficiently small and
    iterating the bound, we deduce that
    \begin{equation}\label{eqn:final-lp}
        \begin{aligned}
            \EE \left[ \sup_{t \leqslant s \leqslant t+1} \| u_{s}
        \|_{L^{p}}^{p} \right] \lesssim \EE [ \| u_{t} \|_{L^{p}}^{p}] \lesssim  e^{ -c_{1} t} \EE \|
        u_{0} \|_{L^{p}}^{p} + 1 \;,
        \end{aligned}
    \end{equation}
    where in the last estimate we used \eqref{eqn:a-priori-lp}. This concludes
    the proof of the a-priori $ L^{p} $ estimate: note that if we could have
    chosen $ p = \infty $ the proposition would be proven.  

    Instead, to obtain the $ L^{\infty} $ estimate we simply bootstrap our
    argument using the Schauder estimates as in \eqref{e:schauder}.
    Indeed, for any $ t, h \geqslant 0 $ we represent the solution $ u_{t + h} $ by
    \begin{align*}
        u_{t + h} = P_{h} u_{t} + \int_{0}^{h} P_{h - r} \left[ \div(A
        (u_{t+r})) + \psi_{0}  \right] \ud r + \sum_{k=1}^{n} \int_{t}^{t+h} P_{t+h -r}
        [\psi_{k}] \ud B^{k}_{r} \;.
    \end{align*}
    Now recall the following Schauder estimates and Sobolev embeddings for $
\alpha \geqslant 0, \alpha \geqslant \beta $ and $ \kappa >0 $:
    \begin{equation}\label{e:sobolev}
\begin{aligned}
        \| P_{t} \varphi \|_{W^{\alpha, p}} & \lesssim t^{- \frac{\alpha}{2}
        } \| \varphi \|_{L^{p}} \;, \quad  \| P_{t} \varphi \|_{W^{\alpha, p}}
\lesssim t^{- \frac{\alpha - \beta}{2}
        } \| \varphi \|_{W^{\beta, p}} \;,\\
\| \varphi \|_{L^{\infty}} & \lesssim \| \varphi \|_{W^{\frac{d}{p} + \kappa ,
p}} \;.
\end{aligned}
    \end{equation}
    Then for \( p \geqslant 1 \) and $ \kappa > 0 $ (to be chosen respectively
    sufficiently large and sufficiently small later on) we find, for $ h > 0 $
    \begin{align*}
        \| u_{t + h} \|_{\infty} \lesssim h^{- \frac{1}{2} (\frac{d}{p}+ \kappa)} \| 
        u_{t} \|_{L^{p}} + \int_{0}^{h} \| P_{h - r} \left[ \div (A
        (u_{t+r})) \right] \|_{W^{\frac{d}{p} + \kappa, p}}  \ud r + \|
        z_{t, h} \|_{L^{\infty}} \;,
    \end{align*}
    where $ z_{t, h} = \int_{t}^{t + h} P_{t +h - r } [\psi_{0}(r, \cdot )] \ud r +
    \int_{t}^{t + h} P_{t+ h-r} [\psi_{k}] \ud
    B^{k}_{r} $. As for the second term, we estimate
    \begin{align*}
        \int_{0}^{h} \| P_{h - r} \left[ \div (A
        (u_{t+r})) \right] \|_{W^{\frac{d}{p} + \kappa, p}}  \ud r \lesssim
        \int_{0}^{h}(h -r)^{- \frac{1}{2} (1 + \frac{d}{p} + 2 \kappa)} \|
        \div(A (u_{t + r})) \|_{W^{-1 - \kappa, p}} \ud r \;.
    \end{align*} %\ana{discuss}
    Hence, assuming that $ p $ is chosen sufficiently large and $ \kappa $
    sufficiently small, so that
    \begin{align*}
        \gamma \stackrel{\mathrm{def}}{=} 1 + \frac{d}{p} + 2 \kappa  < 2 \;,
    \end{align*}
    and estimating
    \begin{align*}
        \| \div(A (u_{t + r})) \|_{W^{- 1 - \kappa, p}} \lesssim \|
        A(u_{t + r}) \|_{L^{p}} \lesssim \| u_{t + r}
        \|_{L^{\mf{a}p}}^{\mf{a}} \;,
    \end{align*}
    we obtain
    \begin{equation}\label{eqn:final-l-inf}
        \begin{aligned}
            \sup_{0 \leqslant s \leqslant h} s^{\gamma /2} \| u_{t + s} \|_{\infty} \lesssim  \| 
        u_{t} \|_{L^{p}} + h \sup_{0 \leqslant s \leqslant h } \| u_{t + s}
        \|_{L^{\mf{a} p}}^{\mf{a}} + \sup_{0 \leqslant s \leqslant h} \|
        z_{t, s} \|_{\infty} \;.
        \end{aligned}
    \end{equation}
    Combining \eqref{eqn:final-lp} with \eqref{eqn:final-l-inf},
    together with our estimates on $ z_{t, s} $ from Lemma~\ref{lem:reg-z} we can
    conclude.

\end{proof}

As a by-product of the proof of
Proposition~\ref{prop:wp} we have actually proven the following statement.
\begin{corollary}\label{cor:pathwise-bound}
    Assume that $ \mf{a} = \infty $. For almost every $ \omega \in \Omega $ there exists a constant $
    \mf{ c} (\omega) \in (0, \infty) $ such that
    \begin{align*}
        \sup_{u_{0} \in L^{\infty}} \sup_{1 \leqslant s \leqslant 2} \|
        u_{s}(\omega) \|_{\infty} < \mf{c}(\omega) \;.
    \end{align*}
\end{corollary}

\section{Synchronisation}

Throughout this section we consider, for two fixed initial conditions $
u_{0}, v_{0} \in L^{\varrho} $, where $ \varrho (d,\mf{a}) \in [\mf{ a}, \infty]$ is chosen as in
Proposition~\ref{prop:wp}, the difference
\begin{align*}
    w(t,x) = u(t, x) - v(t,x) \;, \quad \forall (t, x) \in [0, \infty)
    \times \mD \;,
\end{align*}
where $ u_{t}, v_{t} $ are the solutions to \eqref{eqn:main} with
respectively initial conditions $ u_{0}, v_{0} $. Then \( w \) solves
\begin{equation}\label{eqn:w}
\partial_{t} w  + \div(B(u, v) w) - \Delta w=0\;, \qquad w
\vert_{(0, \infty) \times \partial \mD} = 0 \;,
\end{equation}
with initial condition $ w_{0} = 0 $ and where
\begin{align*}
     B(u, v) = \frac{A(u) - A(v)}{u - v}\;, \qquad B \in C(\RR^{2} ;
     \RR^{d})\;.
\end{align*}
Here the fact that $ B \in C(\RR^{2}; \RR^{d}) $ follows from the requirement
that $ A \in C^{1} $ in Assumption~\ref{assu:noise}. It will be convenient to
consider solutions to \eqref{eqn:w} for general initial conditions $
w_{0} $. In this case we write
\begin{align*}
    (t, x) \mapsto \Phi (u, v, w_{0}; t,x)
\end{align*}
for the solution to \eqref{eqn:w} with initial condition $ w_{0} \in
L^{1} $ and driven
by paths $ u, v \in C([0, \infty) ; L^{\varrho}) $, with $ \varrho \geqslant \varrho
(\mf{a}, d) $, the latter as in Proposition~\ref{prop:wp}.

\begin{lemma}\label{lem:wp-w}
	There exists a $ \overline{\varrho} (\mf{a}, d) \geqslant
\varrho(\mf{a}, d)  $ (the latter as in Proposition~\ref{prop:wp}), such that
for every $ u, v \in C([0, \infty) ; L^{ \overline{\varrho}(\mf{a}, d)}) $ and
every $ w_{0} \in L^{1} $ there exists a unique solution $ w$ to
\eqref{eqn:w}, which can in addition be repsented by
\begin{align*}
w(t, x) = \int_{\mD} \Gamma (u,v;t, x, y) w_{0}(y) \ud y \;, \qquad \forall t
>0\;.
\end{align*}
where $ \Gamma (u, v ; \cdot) \in C( (0, \infty) \times \mD \times \mD ) $.
\end{lemma}
\begin{proof}
Clearly it suffices to construct the fundamental solution $ (t, x)
\mapsto\Gamma(u,v;t, x, y) $ , which solves \eqref{eqn:w} with initial
condition the Dirac delta centered at $ y $:  $ x \mapsto \delta_{y} (x) $.
Note that by Assumption~\ref{assu:noise}, in particular via our growth
requirement on $ A $ and $ A^{\prime} $ , we can find for any $ \zeta > 0 $ a $ \overline{\varrho}
(\mf{a}, d) > 0 $ such that $ B(u, v) \in C([0, \infty); L^{\zeta} )$ if $ u, v \in
C([0, \infty); L^{ \overline{\varrho}(\zeta)} ) $. Then the fundamental
solution can be constructed via classical arguments, by smoothing the initial
data from the Besov space $ \mB^{0}_{\infty, 1} $ to $ L^{\zeta^{\prime}} $,
with $ \zeta^{-1} + (\zeta^{\prime})^{-1} =1 $, see for example
\cite[Chapter 6]{GubinelliPerkowski17KPZReloaded}.
%\tommaso{Arrived here!}
\end{proof}
The property which eventually delivers synchronisation is an $
L^{1} $ contraction principle for \eqref{eqn:main}, whose proof is elementary
and relies on the following computation:
\begin{equation}\label{eqn:div}
\begin{aligned}
    \int_{\mD} w_{t} \ud x & = \int_{0}^{t} \int_{\mD} \Delta w_{s} +
    \div( B(u_{ s},v_{s}) w_{s}) \ud x \ud s + \int_{\mD} w_{0} \ud x \\
    & =\int_{0}^{t} \int_{\partial \mD} \mathbf{n} \cdot ( \nabla w_{s} +
    B(u_{s}, v_{s}) w_{s})  \ud \Sigma  \ud s + \int_{\mD} w_{0} \ud x \\
    & = \int_{0}^{t} \int_{\partial \mD} \mathbf{n} \cdot  \nabla w_{s}   \ud
    \Sigma  \ud s + \int_{\mD} w_{0} \ud x \;,
\end{aligned}
\end{equation}
by the divergence theorem. %, where as before $ \mathbf{n}(x) $ is the unit outer normal at the
%point $ x $ on the boundary  $ \partial \mD $ of $ \mD $ and $ \Sigma $ is the
%Hausdorff measure on the boundary set. 
 In the last line we have used that $ w
\vert_{\partial \mD} = 0 $. Now we observe that if $ w_{0} \geqslant 0 $,
then by a maximum principle $ w_{t} \geqslant 0 $ for any $ t \geqslant 0 $ and
in addition (see Lemma~\ref{lem:contract} below for a proof)
\begin{equation}\label{eqn:to-bound}
    \begin{aligned}
        \int_{\partial \mD} \mathbf{n} \cdot \nabla w_{t}  \ud \Sigma \leqslant  0 \;,
    \end{aligned}
\end{equation}
which from \eqref{eqn:div} implies that for $ w_{0} \geqslant 0 $ we have 
\begin{equation}\label{eqn:l1-first}
    \begin{aligned}
        \int_{\mD} w_{t}(x) \ud x \leqslant \int_{\mD} w_{0}(x) \ud x \;.
    \end{aligned}
\end{equation}
From this we deduce the contraction principle as follows.
\begin{lemma}\label{lem:l1-contraction}
    Under Assumption~\ref{assu:noise}, for any $ u_{0}, v_{0} \in L^{\varrho}
    $, with $ \varrho (\mf{a}, d) $ as in Proposition~\ref{prop:wp} and $
    u_{t},v_{t} $  solutions to \eqref{eqn:main} with respectively
    initial conditions $ u_{0} $ and $ v_{0} $, and $ w_{t} = u_{t} -
    v_{t} $, for any $ t \geqslant 0 $ we have that $\| w_{t} \|_{L^{1}}
    \leqslant \| w_{0} \|_{L^{1}}$.
\end{lemma}

\begin{proof}
    We decompose $ w_{0} = (w_{0} )_{+} - (w_{0})_{-} $ in its positive and
    negative part, namely with $ x_{+} = \max \{ x, 0 \} $ and $ x_{-} = 
    \max \{ -x, 0 \} $. Then by linearity of \eqref{eqn:w} we have $ w_{t} =
    w_{t}^{+} - w_{t}^{-} $, where $ w^{\pm}_{t} $ is the solution to
    \eqref{eqn:w} with initial condition $ (w_{0})_{\pm} $. Then by a maximum
    principle we have $ w_{t}^{\pm} \geqslant 0 $. In addition, by
    \eqref{eqn:l1-first} we conclude that
    $ \| w_{t} \|_{L^{1}} \leqslant \|
    w_{t}^{+} \|_{L^{1}} + \| w_{t}^{-} \|_{L^{1}} \leqslant  \| w_{0}^{+} \|_{L^{1}} + \|
    w_{0}^{-} \|_{L^{1}} = \| w_{0} \|_{L^{1}}$.
\end{proof}
In particular, the essence of our proof of synchronisation is to show that a strict contraction
happens with positive probability, namely that $ \| w_{t} \|_{L^{1} } < \zeta
\| w_{0} \|_{L^{1}} $ for some random and shift-ivariant $ \zeta \in [0, 1] $ such that $ \PP(\zeta < 1)
> 0$ (actually the argument in the case $ \mf{a} < \infty $ will be slightly
more involved, as the $ \zeta $ we obtain will not be shift-invariant). For
this purpose, we will need a quantitative bound on the flux at the boundary
\eqref{eqn:to-bound}.

To state the next result, conisider the solution $ [0, \infty) \ni s
\mapsto X_{s}^{t} \in \RR^{d} $ to the SDE 
\begin{equation}\label{e:sde}
    \ud X_{s}^{t}  = B(u_{t+s}, v_{t+s}) (X_{s})\ud s + \ud W_{s}\;,
\end{equation}
where $ W $ is a  $ d $--dimensional Brownian motion,
and $ \tau $ is the stopping time $ \tau = \inf \{ t \geqslant 0  \; \colon
\; X_{s}^{t} \in \partial \mD \} $: the solution to the SDE is
    defined only up to time $ \tau $, so for simplicity we define $
X_{s}^{t} = X_{\tau}^{t} $ for $ \tau \geqslant s $. Since both $ u $ and $v $ lie in \(
C([t, t+h] \times \overline{\mD}) \) for any \( t,h > 0 \) in both cases $
\mf{a} < \infty $ and $ \mf{a} = \infty $ by Proposition~\ref{prop:wp}, it is
straightforward to see that \eqref{e:sde} admists a unique solution. In
particular, the quantity 
\begin{equation}\label{e:pth}
    \mathbf{p}_{t, t+h} =  \inf_{y \in \mD} \PP_{y} (X_{h \wedge
    \tau}^{t} \in \partial \mD)\in (0, 1) 
\end{equation}
is well defined for any $ t, h > 0 $. The next result tells us that $ \mathbf{p}_{t, t+h} $ bounds
the dissipation at the boundary of solutions to \eqref{eqn:w}.

\begin{lemma}\label{lem:contract}
    Under Assumption~\ref{assu:noise} consider $ u_{0}, v_{0} \in L^{\varrho}
    $, for $ \varrho(\mf{a}, d) \in [\mf{a}, \infty] $ as in
    Proposition~\ref{prop:wp} and $ u_{t}, v_{t} $ the solutions to
    \eqref{eqn:main} with respectively intial condition $ u_{0} $ and $
    v_{0} $. Moreover, for any non-negative initial condition $
    \Phi_{0} = w_{0} \geqslant 0, w_{0} \in L^{\varrho} $, let $ (t, x) \mapsto \Phi(u,v, w_{0}; t, x)
    $ be the solution to \eqref{eqn:w} with initial condition $
    w_{0} $. Then for any $ t, h > 0 $ and $ \omega \in \Omega $ there exists a $
    \mathbf{p}_{t, t+h} (\omega,
    u_{0},v_{0}) \in (0, 1) $ such that
\begin{align*}
    \int_{t}^{t+h} \int_{\partial \mD} \mathbf{n}(x) \cdot \nabla \Phi( u , v, w
    _{0}; s, x)  \ud \Sigma (x) \ud s \leqslant
    - \mathbf{p}_{t, t+h} \| w_{0} \|_{L^{1} } \;.
\end{align*}
\end{lemma}
\begin{proof}

    We can represent the solution $ w
    $ through the kernel $ \Gamma $ associated to \eqref{eqn:w}. Namely for any
    $ y \in \mD $, let $ \Gamma ( u, v; t, x, y) = \Phi (u,v, \delta_{y}; t, x)$ be
    the solution to \eqref{eqn:w} with
    initial condition $ \Gamma (u,v; 0, x, y) = \delta_{y} (x) $. Then by linearity
    of \eqref{eqn:w} we have $ \Phi (u,v, w_{0}; t, x) = \int_{\mD} \Gamma
    (u,v; t, x ,y) w_{0}(y) \ud y $ for any initial
    condition $ w_{0} $. To simplify the notation we shall henceforth write $
    \Gamma (t, x) = \Gamma(u, v; t, x) $ and $ \Phi(t, x) = \Phi (u, v,
    w_{0}; t, x) $.
    Then we can bound
\begin{align*}
    \int_{t}^{t+h} \int_{\partial \mD} \mathbf{n}(x) \cdot \nabla_{x} w_{t}(x)  \ud \Sigma(x) \ud
    s & = \int_{t}^{t+h} \int_{\mD} \int_{\partial \mD} \mathbf{n} (x) \cdot \nabla_{x} \Gamma_{s} (x, y)
    w_{0}(y) \ud \Sigma(x) \ud y \ud s \\
    & \leqslant \left\{ \sup_{y \in \mD}
        \int_{t}^{t+h} \int_{\partial \mD} \mathbf{n} (x)
    \cdot \nabla_{x} \Gamma_{s} (x, y) \ud \Sigma(x) \ud s \right\} \| w_{0}
    \|_{L^{1}}\;.
\end{align*}
If we follow backwards the calculations in \eqref{eqn:div} we find
\begin{align*}
    \sup_{y \in \mD} \left\{ \int_{t}^{t+h} \int_{\partial \mD} \mathbf{n}
    \cdot \nabla_{x} \Gamma_{s} (x, y) \ud \Sigma(x) \ud s \right\} = \sup_{y \in \mD} \left\{
    \int_{\mD} \Gamma_{t}(x , y) \ud x - 1 \right\}
    = - \mathbf{p}_{t, t+h} \;.
\end{align*}
The proof is complete.
\end{proof}

Note that $ \mathbf{p}_{t, h} $ is decreasing in $ t $ and increasing in $
h $. In particular by continuation we can define $ \mathbf{p}_{t, h} \in [0, 1]
$ for all $ t, h \geqslant 0 $.

\begin{corollary}\label{cor:strict-contraction}
    Under Assumption~\ref{assu:noise} consider $ u_{0}, v_{0} \in L^{\varrho}
    $, for $ \varrho(\mf{a}, d) \in [\mf{a}, \infty] $, and let $ w_{t} =
    u_{t} - v_{t} $. Then for any $ t, h \geqslant  0 $
    \begin{align*}
        \| w_{t+h} \|_{L^{1}} \leqslant (1 - \mathbf{p}_{t, t + h}) \| w_{t}
        \|_{L^{1}} \;.
    \end{align*}
\end{corollary}
\begin{proof}
    It suffices to prove the result for $ t, h > 0 $ and passing to limits in
    the case $ t $ or $ h = 0$. As in the proof of Lemma~\ref{lem:l1-contraction} we have
    $\| w_{t} \|_{L^{1}} \leqslant \| w_{t}^{+} \|_{L^{1}} + \| w_{t}^{-}
    \|_{L^{1}}$. Now in addition by Lemma~\ref{lem:contract} and the calculation in \eqref{eqn:div}
    \begin{align*}
        \| w^{+}_{t+h} \|_{L^{1}} \leqslant (1 - \mathbf{p}_{t, t +h}) \|
        w_{t}^{+} \|_{L^{1}} \;, \qquad \| w^{-}_{t+h} \|_{L^{1}} \leqslant
        (1 - \mathbf{p}_{t, t+h}) \| w_{t}^{-} \|_{L^{1}} \;,
    \end{align*}
    so that the result follows immediately.
\end{proof}

    Next we provide a lower bound to the exit probability $
    \mathbf{p}_{t, t+h} $, for any $ t, h >0 $, via Girsanov's transformation.
%\begin{equation}\label{eqn:p-n}
        %\begin{aligned}
            %\mathbf{p}_{t,t+h}(u, v, \omega) =\inf_{y \in \mD} \PP_{y} (X_{h
                %\wedge \tau}^{t}
            %\in \partial \mD) \in [0, 1] \;,
        %\end{aligned}
    %\end{equation}
    %where under the law $ \PP_{y} $, the precess  $ s \mapsto X_{s}^{t} \in
    %\RR^{d} $ is the solution to the SDE
    %\begin{equation}\label{eqn:X-n}
        %\ud X_{s}^{t} = B(u_{t+s}, v_{t+s}) (X_{s}^{t}) \ud s + \ud W_{s}\;,
        %\qquad X^{n}_{0} = y \;,
    %\end{equation}
    %and $ \tau $ the stopping time $ \tau = \inf \{ s \geqslant 0  \; \colon
    %\; X_{s}^{t} \in \partial \mD \} $. Note that the solution $
    %X^{t}_{s} $ is defined only until the stopping time $ \tau $ (since the
    %drift is defined only in the domain $ \mD $). For convenience we set $
    %X^{t}_{s} = X^{t}_{\tau} $ for all $ s \geqslant \tau $.

    \begin{lemma}\label{lem:control}
    Under Assumption~\ref{assu:noise} there exist constants $ c \in
    (0, 1), C > 0 $ depending only on the choice of the domain $ \mD $, such that for
    all $ t \geqslant 0\;, h \in [0, 1] $
    \begin{align*}
        \mathbf{p}_{t, t+h}(u, v, \omega)  \geqslant c \exp \left\{ - h^{-1}
            \left( C +  \sup_{0 \leqslant s \leqslant h}  \| B (u_{t+
        s}(\omega), v_{t+s}(\omega)) \|_{\infty}^{2} \right) \right\} \;,
    \end{align*}
    where $ \mathbf{p}_{t, t+h} $ is as in \eqref{e:pth}.

%    $ \zeta, \ve \in (0, 1) $
    %such that for \( \mathbf{p}_{n, n+1} \) as in \eqref{eqn:p-n}
    %\begin{align*}
        %\PP ( \mathbf{p}_{n, n+1} > \zeta) > \ve \;, \qquad \forall n \in \NN \;.
    %\end{align*}
\end{lemma}
\begin{proof}
    This result follows from Girsanov's transformation. Fix the drift $ (s,x)
    \mapsto \psi^{t}_{s} (u, v)(x) $ defined by
    \begin{align*}
        \psi^{t}_{s}(x)  = - B(u_{t+s}, v_{t+ s})(x) + C h^{-1} e_{1} \;, \qquad \forall s
        \in [0, h] \;,
    \end{align*}
    for $ x \in \mD $, with $ e_{1} = (1,0, \dots, 0) \in \RR^{d} $ and $ C > 0 $ a positive
    constant to be chosen later on.
    Then consider the ode
    \begin{align*}
        \partial_{s} X^{t, \psi^{t}}_{s} = B(u_{t+s}, v_{t+s})(X_{s}^{t,
        \psi^{t}}) + \psi_{s}^{t}(X^{t, \psi^{t}}_{s}) \;, \forall s \in
        [0,h] \;,
    \end{align*}
    so that $ X^{t, \psi^{t}}_{h} = X_{0}^{t, \psi^{t}} + C e_{1} \not\in \mD $ for all $
    X_{0}^{n, \psi^{n}} \in \mD $, provided $ C > 0 $ is chosen sufficiently
    large, depending on the domain \( \mD \).

    Now, since $ \psi^{t} \in \mathcal{CM} $, the Cameron--Martin space of $ W
    $, we find that under the probability measure $ \PP^{\psi^{t}}_{y} $ given by the
    Radon--Nikodym derivative
    \begin{align*}
        \frac{\ud \PP^{\psi^{t}}_{y}}{\ud \PP_{y}} = \exp \left\{ \int_{0}^{h}
            \psi_{s}^{t} \ud W_{s} - \frac{1}{2} \int_{0}^{h} | \psi_{s}^{t} |^{2} \ud s \right\} \;,
    \end{align*}
    the solution $ X_{s}^{t} $ to the SDE \eqref{e:sde} solves the SDE
    \begin{align*}
        \ud X_{s}^{t} = C e_{1} \ud s + \ud \widetilde{W}_{s} \;, \qquad
        X^{t}_{0} = y \;, \qquad \forall s \in [0, h \wedge \tau] \;,
    \end{align*}
    where $ \widetilde{W} $ is a Brownian motion under $ \PP^{\psi^{t}}_{y} $.
    In particular, under $ \PP^{\psi^{t}}_{y} $ we have $X_{h}^{t} = y + \widetilde{W}_{h} + C
    e_{1}$. Hence if $ C > 0 $ is sufficiently large so that $ (\mD+
    B_{1}(0)+ C e_{1}) \cap \mD = \emptyset$ (here $ B_{1} (0) $ is the ball of unit Euclidean
    radius centered about zero and we consider the sum of sets $ A + B = \{
    a + b  \; \colon \; a \in A , b \in B \} $), we can conclude that
    \begin{equation}\label{eqn:lb-p}
        \begin{aligned}
            \inf_{y \in \mD} \PP^{\psi^{t}}_{y}( X_{h}^{t} \not\in  \mD)
            \geqslant \PP^{\psi^{t}}_{y}( |
            \widetilde{W}_{h} | \leqslant 1 ) \geqslant \PP^{\psi^{t}}_{y}( |
            \widetilde{W}_{1} | \leqslant 1 )= c\;,
        \end{aligned}
    \end{equation}
    for some constant $ c > 0 $, where we used that $ h \in [0, 1] $. Now we
    would like to pass from the probability
    measure $ \PP^{\psi^{t}}_{y} $ to the original probability measure $
    \PP_{y} $:
    \begin{align*}
        \PP^{\psi^{t}}_{y}(A) & = \EE_{y} \left[ 1_{A}\exp \left\{ \int_{0}^{h}
                \psi^{t}_{s}(X^{t}_{s})  \ud
        W_{s} - \frac{1}{2} \int_{0}^{h} | \psi^{t}_{s}(X^{t}_{s}) |^{2} \ud s \right\}\right] \\
        & \leqslant \PP_{y}(A)^{\frac{1}{2}} \EE_{y} \left[\exp \left\{ 2 \int_{0}^{h}
                \psi^{t}_{s}(X^{t}_{s}) \ud W_{s} -  \int_{0}^{h} |
                \psi^{t}_{s}(X^{t}_{s}) |^{2} \ud s \right\}\right]^{\frac{1}{2}} \\
        & \leqslant \PP_{y} (A)^{\frac{1}{2}} \EE_{y} \left[\exp \left\{  \int_{0}^{h} |
                \psi^{t}_{s}(X^{t}_{s}) |^{2} \ud s \right\}\right]^{\frac{1}{2}} \;.
    \end{align*}
    In addition, we have that
    \begin{align*}
        \sup_{x \in \mD}| \psi^{t}_{s}(x)| \leqslant \sup_{0 \leqslant s \leqslant h}
        \| B(u_{t+s}, v_{t+s}) \|_{\infty} + C h^{-1}  \;.
    \end{align*}
    We therefore conclude 
    \begin{align*}
        \inf_{y \in \mD} \PP_{y} (X^{t}_{h} \not\in \mD) & \geqslant \left[\inf_{y \in \mD}
        \PP_{y}^{\psi^{t}} (X^{t}_{h} \not\in \mD) \right]^{2} \left\{
        \EE_{y} \left[\exp \left\{  \int_{0}^{h} |
        \psi^{t}_{s}(X^{t}_{s}) |^{2} \ud s \right\}\right]\right\}^{-1} \\
        & \geqslant c^{2} \; \EE_{y} \left[\exp \left\{  \int_{0}^{h} - |
        \psi^{t}_{s}(X^{t}_{s}) |^{2} \ud s \right\}\right] \\
        & \geqslant c^{2} \exp \left\{ - h^{-1} \left( C + \sup_{0 \leqslant s \leqslant h}
        \| B (u_{t + s}, v_{t+ s}) \|_{\infty}^{2} \right) \right\} \;,
    \end{align*}
    with $ c $ as in \eqref{eqn:lb-p} and by Jensen's inequality. This
    concludes the proof of the lemma, up to choosing a smaller $ c > 0 $.

    %\begin{align*}
        %c_{1} = \exp \left\{ - \frac{1}{2}  \left( \sup_{0 \leqslant s \leqslant 1} \| B(u_{s},
        %v_{s}) \|_{\infty} + C  \right)^{2} \right\} \;.
    %\end{align*}
\end{proof}

We want to make use of the previous result in combination with the ergodic
theorem. Therefore, let us write $ \mf{S}_{t}(\omega) $ for the (random)
solution map to \eqref{eqn:main} at time $ t \geqslant 0 $: namely $ u (t,
\omega) = \mf{S}_{t} (\omega) u_{0} $. Then we can bound, for every $ n \in \NN
\setminus \{ 0 \} $:
\begin{equation}\label{eqn:c-B}
    \begin{aligned}
            \sup_{0 \leqslant s \leqslant 1} \| B(u_{n+s}, v_{n+s} ) \|_{\infty}
    & \leqslant \sup_{u_{0}, v_{0} \in L^{\infty}} \| B( \mf{S}_{1 +
    s}(\vt^{n-1} \omega) u_{0}, \mf{S}_{1 + s}(\vt^{n-1} \omega) v_{0})
    \|_{\infty} \\
    & \leqslant \sup_{a, b \in B_{\mf{c}(\vt^{n-1} \omega)}(0)} B(a, b)
    \stackrel{\mathrm{def}}{=} \mf{c}_{B} (\vt^{n -1}\omega) < \infty \;,
    \end{aligned}
\end{equation}
with $ \mf{c} (\omega) $ as in Corollary~\ref{cor:pathwise-bound} and where $
B_{\varrho} (0) $ is the ball of radius $ \varrho > 0 $ about zero.

\begin{proof}[Proof of Theorem~\ref{thm:synchro}]
    As usual we distinguish the case $ \mf{a} = \infty $ from the case $
    \mf{a} < \infty$.

    \textit{Case $ \mf{a} = \infty $.} Here we can use Lemma~\ref{lem:contract}
    together with \eqref{eqn:div} (and following the notation in the proof of
    Lemma~\ref{lem:contract}) to bound, for every $ n \in \NN \setminus \{
    0 \} $ 
    \begin{align*}
        \| w_{n+1} \|_{L^{1}} & \leqslant \| w^{+}_{n+1} \|_{L^{1}} + \|
        w^{-}_{n+1}  \|_{L^{1}} \\
        & \leqslant (1 - \mathbf{p}_{n, n+1}) ( \| w_{n}^{+}
        \|_{L^{1}} + \| w_{n}^{-} \|_{L^{1}})  \leqslant (1 - \mathbf{p}_{n,
        n+1}) \| w_{n} \|_{L^{1}} \;,
    \end{align*}

    Therefore we obtain
    \begin{align*}
        \lim_{n \to \infty} \frac{1}{n} \log{ \| w_{n} \|_{L^{1}}} \leqslant
        \lim_{n \to \infty} \frac{1}{n} \left( \sum_{i=0}^{n-1} \log{ (1 -
        \mathbf{p}_{i, i+1})} + \log{ \| w_{0} \|_{L^{1}}} \right)\;.
    \end{align*}
    Now we use the bound in Lemma~\ref{lem:control} together with
    \eqref{eqn:c-B} to obtain by the ergodic theorem, almost surely
    \begin{align*}
        \lim_{n \to \infty} \frac{1}{n} \log{ \| w_{n} \|_{L^{1}}} & \leqslant
        \lim_{n \to \infty} \frac{1}{n}  \sum_{i=1}^{n-1} \log{\left(1 - \exp
        \left\{ - \mf{c}_{B}^{2} (\vt^{i-1} \omega)  \right\} \right)} \\
        & = \EE \left[ \log{\left(1 - \exp \left\{ - \mf{c}_{B}^{2} \right\} \right)}
        \right] < 0 \;. 
    \end{align*}
    Since almost surely
    \begin{align*}
        \lim_{n \to \infty}\frac{1}{n} \left\{\log{ (1 - \mathbf{p}_{0, 1})}+
        \log{\| w_{0} \|_{L^{1}} }  \right\} = 0\;,
    \end{align*}
    which concludes the proof.

    \textit{Case $ \mf{a} < \infty $.} Here we use the results of
    Section~\ref{sec:lyap} below. In that section we introduce a constant $
		\overline{R}_{2} \in (0, \infty)$ and a sequence of stopping times
		\begin{align*}
				0 \leqslant \tau^{\mathrm{in}}_{0} < \dots < \tau^{\mathrm{in}}_{i} <
				\tau^{\mathrm{out}}_{i} \;,
		\end{align*}
		such that $ \max \{ \| u_{t} \|_{\infty}, \| v_{t} \|_{\infty} \} \leqslant
		\overline{R}_{2} $ for $ t \in [\tau^{\mathrm{in}}_{i},
		\tau^{\mathrm{out}}_{i}] $ and $ i \in \NN $. Then let us define 
		\begin{align*}
				B= B ( \overline{R}_{2}) = \max \{ B(u, v)  \; \colon \; \max \{ | u |, | v |\}
				\leqslant \overline{R}_{2} \} \;.
		\end{align*}
    Moreover, by construction $ T_{i}^{\infty} =
    \tau^{\mathrm{out}}_{i} - \tau^{\mathrm{in}}_{i} \leqslant 1$, so that we
    are in the setting of Lemma~\ref{lem:control} and can bound
    \begin{align*}
        \log{ \| w_{t} \|_{L^{1}} } & \leqslant \sum_{i  \; \colon \;
        \tau^{\mathrm{out}}_{i} \leqslant t}
        \log{\mathbf{p}_{\tau^{\mathrm{in}}_{i}, \tau^{\mathrm{out}}_{i}} } +
        \log {\| w_{\tau^{\mathrm{in}}_{1}} \|_{L^{1}}}
        \leqslant - \sum_{i  \; \colon \;
        \tau^{\mathrm{out}}_{i} \leqslant t}
        c(T_{i}^{\infty}, B) +
        \log {\| w_{\tau^{\mathrm{in}}_{1}} \|_{L^{1}}} \;,
    \end{align*}
    where the constants $ c(T_{i}^{\infty}, B) $ are given 
    by
    \begin{align}\label{eqn:def-c}
        c(T^{\infty}_{i}, B) = - \log \left\{ 1 - c\exp \left(
        (T_{i}^{\mathrm{\infty}})^{-1} (C + B)   \right) \right\} \;,
    \end{align}
    with $ c , C > 0 $ defined in Lemma~\ref{lem:control}. Now by
    Proposition~\ref{prop:center-time} there exits an $ \eta= \eta
    (B) > 0 $ such that
    \begin{align*}
        \limsup_{t \to \infty} \frac{1}{t}  \log{\| w_{t} \|_{L^{1}}} \leqslant
        - \eta \;,
    \end{align*}
    so that the result is proven.

\end{proof}

\subsection{Lyapunov estimates}\label{sec:lyap}

In this section we consider the setting of Assumption~\ref{assu:noise} in the
case $ \mf{a} < \infty $. We then choose two initial condition $ u_{0},
v_{0} \in L^{\varrho} $, for $ \varrho (\mf{a}, d) $ as in
Proposition~\ref{prop:wp}, and associated respectively to two solutions $
u_{t}, v_{t} $ of \eqref{eqn:main} driven by the same noise. The aim of this section is then to derive
estimates on the time that the couple \( \mathbf{u}_{t} = (u_{t}, v_{t}) \)
spends in the ``center'' of the state space, i.e.\ a ball of finite volume
about the origin in $ L^{\varrho} $. Here we will make use of the last
requirement in Assumption~\ref{assu:noise}, that $ \psi_{0} $ is not random, so
that the process $ \mathbf{u}_{t} $ is Markov.

For a 2D vector $ \mathbf{u} = (u, v) \in L^{p}(\mD;
\RR^{2}) $ we will make use of the norms $ \| \mathbf{u} \|_{L^{p}} = (\| u
\|_{L^{p}}^{p} + \| v \|_{L^{p}}^{p} )^{\frac{1}{p}}$ for $ p \in [1, \infty],$
with the usual convention for $ p = \infty $.
We fix any $ 0 < R_{1} < R_{2} < R_{3} < \infty $ (we shall fix these values in
Lemma~\ref{lem:moment-bound}) and define the ``center'' subsets of $
L^{p} $
\begin{equation}\label{eqn:center-Lp}
    \begin{aligned}
            \tC^{-} = \{ \mathbf{u} \in L^{p}  \; \colon \;\| \mathbf{u}
    \|_{L^{p}} \leqslant R_{1} \} \;, \ \
    \tC^{\mathrm{m}} = \{ \mathbf{u} \in L^{p}  \; \colon \;\| \mathbf{u}
    \|_{L^{p}} \leqslant R_{2} \} \;, \\
    \tC^{+} = \{ \mathbf{u} \in
        L^{p}  \; \colon \; \| \mathbf{u} \|_{L^{p}} \leqslant R_{3} \}\;.
    \end{aligned}
\end{equation}
Our aim will then be to prove that the process $ \mathbf{u}_{t} $ passes a
substantial amount of time in the set $ \tC^{+} $: this will guarantee 
an $ L^{1} $ contraction with rate roughly dependent on $ R_{3} $. Actually, since the contraction
depends on the $ L^{\infty} $ norm of $ \mathbf{u}_{t} $, we will have to
additionally make use of parabolic regularisation to show that once in $
\tC^{+}$, the process is likely to stay bounded also in $ L^{\infty} $. The
reason why we consider the $ L^{p} $ norm in the definition of $
\tC^{\pm} $ is that in the energy estimates of Proposition~\ref{prop:wp} we have
shown that if $ \mf{a} < \infty $, the $ L^{p} $ norms are Lyapunov functionals for
the process $ u_{t} $, so our estimates on the time passed in $ \tC^{+} $ will
follow from classical bounds on return times for Markov processes. In any case,
for some $ 0 < \overline{R}_{1} < \overline{R}_{2} < \infty $ to be chosen
later on, define also the following subsets of $ L^{\infty}$ 
\begin{equation}\label{eqn:center-Linf}
    \begin{aligned}
    \tC^{-}_{\infty} = \{ \mathbf{u} \in L^{\infty}  \; \colon \; \|
    \mathbf{u} \|_{\infty} \leqslant \overline{R}_{1} \} \;, \qquad
    \tC^{+}_{\infty} = \{ \mathbf{u} \in L^{\infty}  \; \colon \; \|
    \mathbf{u} \|_{\infty} \leqslant \overline{R}_{2} \} \;.
    \end{aligned}
\end{equation}
Then we consider two types of excursions: excursions in the center and from the
center.
Suppose that $ \mathbf{u}_{0} \not\in \tC^{-}$ (otherwise we would have to define
the stopping times starting from $ \sigma_{0} =0$), set $ \tau_{0} = 0 $ and then
define for $ i \in \NN $
\begin{align*}
    \sigma_{i} & = \begin{cases} \inf \{ t \geqslant \tau_{i}  \; \colon \;
        \mathbf{u}_{t} \in \tC^{-} \} & \text{ if } \mathbf{u}_{\tau_{i}} \not\in
        \tC^{\mathrm{m}} \;, \\
        \tau_{i} & \text{ if } \mathbf{u}_{\tau_{i}} \in \tC^{\mathrm{m}}\;,
    \end{cases} \\
    \tau_{i+1} & = \inf \{ t \geqslant \sigma_{i} \; \colon \; \mathbf{u}_{t}
    \not\in \tC^{+} \} \wedge (\sigma_{i} + 1) \;, \\
    \tau_{i+1}^{\mathrm{in}} & = \inf \{ t \geqslant \sigma_{i}  \; \colon \;
    \mathbf{u}_{t} \in \tC^{-}_{\infty} \} \wedge \tau_{i+1} \;, \\
    \tau_{i+1}^{\mathrm{out}} & = \inf \{ t \geqslant
    \tau_{i+1}^{\mathrm{in}}  \; \colon \; \mathbf{u}_{t} \not\in
    \tC^{+}_{\infty } \} \wedge \tau_{i+1} \;.
\end{align*}
Then $ \{ S_{i} \stackrel{\mathrm{def}}{=} \sigma_{i} - \tau_{i}  \}_{i
\in \NN} $ is the succession of outer excursions lengths
(return times from $ (\tC^{\mathrm{m}})^{c} $ to $ \tC^{-} $) and $ \{ T_{i}
\stackrel{\mathrm{def}}{=} \tau_{i} - \sigma_{i-1} \}_{i \in \NN \setminus \{
0 \}} $ is the succession of internal excursion lengths, roughly from $ \tC^{-} $ to $
(\tC^{\mathrm{m}})^{c} $: we have introduced the middle shell $
\tC^{\mathrm{m}} $ so that for later convenience we have the additional
property $ T_{i} \leqslant 1 $ (otherwise we could have simply defined $
\tau_{i+1} = \inf \{ t \geqslant \sigma_{i} \; \colon \; \mathbf{u}_{t}
\not\in \tC^{+} \} $ and the second case in the definition of $
\sigma_{i} $ would not be necessary, as well as the distinction between $
\tC^{\mathrm{m}} $ and $ \tC^{+} $). Similarly we consider $ \{
T_{i}^{\infty} \stackrel{\mathrm{def}}{=} \tau_{i}^{\mathrm{out}} -
\tau_{i}^{\mathrm{in}} \}_{i \in \NN \setminus \{ 0 \}} $.
%Then $ \{ S_{i} \}_{i \in \NN} $ is a sequence of
%independent (but not identically distributed) random variables. 
For later convenience, let us also define the following stopping times, for a
given $ \mathbf{u} \in L^{p} $ (here $ \mathbf{u}_{t} $ is
the evolution of the Markov process $ (u_{t}, v_{t}) $ with initial condition $
\mathbf{u}_{0} = \mathbf{u}  $):
\begin{align*}
    & S(\mathbf{u}) =  \begin{cases} \inf \{ t \geqslant 0  \; \colon \;
        \mathbf{u}_{t} \in \tC^{-} \} & \text{ if } \mathbf{u} \not\in
        \tC^{\mathrm{m}} \;, \\
        0 & \text{ if } \mathbf{u} \in \tC^{\mathrm{m}}\;,
    \end{cases}\;, \qquad & & \text{ for } \mathbf{u} \in
    (\tC^{-})^{c} \;, \\
    & T (\mathbf{u})  = \inf \{ t \geqslant 0  \; \colon \; \mathbf{u}_{t} \not\in 
    \tC^{+}\} \wedge 1 \;, \qquad & & \text{ for } \mathbf{u} \in
    \tC^{\mathrm{m}} \;, \\
    &T^{\mathrm{in}} (\mathbf{u})  = \inf \{ t \geqslant 0 \; \colon \;
    \mathbf{u}_{t} \in \tC^{-}_{\infty} \} \wedge T (\mathbf{u}) \;, & & \text{
    for } \mathbf{u} \in \tC^{\mathrm{m}} \;, \\
    &T^{\mathrm{out}} (\mathbf{u})  = \inf \{ t \geqslant
        T^{\mathrm{in}}(\mathbf{u}) \; \colon \; \mathbf{u}_{t} \not\in
    \tC^{+}_{\infty } \} \wedge T(\mathbf{u}) \;, & & \text{ for }
    \mathbf{u} \in \tC^{\mathrm{m}} \;.
\end{align*}
Finally define for any $ B > 0 $
\begin{align}\label{eqn:def-L}
    X_{t} = \sup \Big\{ n \in \NN  \; \colon \; \sum_{i = 0}^{n} S_{i} \leqslant t
    \Big\} \;, \qquad L_{t}(B) = \sum_{i =1}^{X_{t}} c ( T_{i}^{\infty},
    B)\;,
\end{align}
which are respectively the total number of excursions by time $ t 
$ and a functional of the time spent in the $ L^{\infty}  $--norm center: the
constant $ c $ is defined in \eqref{eqn:def-c}. 
We want to prove the following result.

\begin{proposition}\label{prop:center-time}
    Under Assumption~\ref{assu:noise}, in the case $ \mf{a} < \infty $,
    consider the center sets defined in \eqref{eqn:center-Lp} and
    \eqref{eqn:center-Linf} with constants $ p(\mf{a}, d) \in [\mf{a}, \infty)
    $,  $ 0 < R_{1} < R_{2} < R_{3} < \infty $ and
    $ 0 < \overline{R}_{1} < \overline{R}_{2} < \infty $ as in
    Lemma~\ref{lem:moment-bound}.
    Then for any $ B > 0 $ there exists a
    positive (deterministic) $ \eta(B) > 0 $ and a $
    p(\mf{a},d) \in [\mf{ a}, \infty) $ such that for any
    initial condition $ \mathbf{u} \in L^{p} $
    \begin{equation*}
    \begin{aligned}
        \liminf_{t \to \infty} \frac{L_{t}(B)}{t} \geqslant \eta \;, \qquad \PP
        \text{--almost surely} \;,
    \end{aligned}
\end{equation*}
with $ L_{t} (B) $ as in \eqref{eqn:def-L}.
\end{proposition}

\begin{proof}
    By the martingale law of large numbers (from here on MLLN) we find
\begin{equation}\label{eqn:LLN-L}
    \begin{aligned}
        \liminf_{n \to \infty} \frac{1}{n} \sum_{i = 0}^{n} c
        (T_{i}^{\infty}, B) \geqslant \inf_{\mathbf{u}
        \in \tC^{\mathrm{m}}} \EE [ c( T^{\mathrm{out}}(\mathbf{u}) -
        T^{\mathrm{in}} (\mathbf{u}), B) ]\stackrel{\mathrm{def}}{=}
        \gamma_{1} \in [0, \infty] \;.
    \end{aligned}
\end{equation}
In particular, if we can in addition prove the following statement:
\begin{equation}\label{eqn:LLN-X}
    \begin{aligned}
        \liminf_{t \to \infty} \frac{X_{t}}{t} \geqslant
        \inf_{\mathbf{u} \in (\tC^{-})^{c}\cap \tC^{+}}
        \frac{1}{\EE [S (\mathbf{u})]} \stackrel{\mathrm{def}}{=}
        \gamma_{2} \in [0, \infty]\;, \qquad \PP
    \text{--almost surely}\;,
    \end{aligned}
\end{equation}
we would be able to conlude
\begin{align*}
    \liminf_{t \to \infty} \frac{L_{t} (B)}{t} & = \liminf_{t \to \infty}
    \frac{X_{t}}{t} \frac{1}{X_{t}} \sum_{i = 0}^{X_{t}} T_{i}^{\infty} \\
    & \geqslant
    \frac{\inf_{\mathbf{u} \in \tC^{\mathrm{m}}} \EE [ c (
        T^{\mathrm{out}}(\mathbf{u}) - T^{\mathrm{in}} (\mathbf{u}),
        B) ]}{\sup_{\mathbf{u} \in (\tC^{-})^{c}\cap \tC^{+}}
    \EE[S(\mathbf{u})]} = \gamma_{1} \cdot \gamma_{2} \stackrel{\mathrm{def}}{=} \eta \in [0, \infty) \;.
\end{align*}
Finally, to show that $ \eta > 0  $ we would need
\begin{align}\label{eqn:eta}
    \inf_{\mathbf{u} \in \tC^{\mathrm{m}}} \EE [c (T^{\mathrm{out}}(\mathbf{u}) -
    T^{\mathrm{in}} (\mathbf{u}), B)]> 0\;,\qquad \sup_{\mathbf{u} \in (\tC^{-})^{c}\cap
    \tC^{+}} \EE[S(\mathbf{u})] < \infty \;.
\end{align}
Hence, to conclude, we have to check \eqref{eqn:LLN-L}, \eqref{eqn:LLN-X} and
\eqref{eqn:eta}: the latter follows from Lemma~\ref{lem:moment-bound}, so we
restrict to proving the first two.

\textit{Proof of \eqref{eqn:LLN-L}}. We make use of the MLLN applied to the
following sum, with $ \mG_{i} = \mF_{\sigma_{i}} $ for $ i \in \NN $:
\begin{align*}
    P_{n} = \sum_{i = 0}^{n} T_{i}^{\infty} - \nu_{i} \;, \qquad \nu_{i} = \EE
    [ T_{i}^{\infty} \vert \mG_{i}] \;.
\end{align*}
In particular, since $ T_{i}^{\infty} \leqslant T_{i} \leqslant 1 $, we have
\begin{align*}
    \sum_{i = 0}^{\infty} \frac{\EE | T_{i}^{\infty} - \nu_{i} |^{2}}{i^{2}} < \infty
    \;, 
\end{align*}
so that $ \PP $--almost surely and in $ L^{1} $ we obtain the convergence $n^{-1}
P_{n} \to 0$. Hence, since by the Markov property of $ \mathbf{u}_{t} $ we have
$\nu_{i} \geqslant \gamma_{1}$ for each $ i \in \NN $, $ \PP $--almost surely,
the claim follows.

\textit{Proof of \eqref{eqn:LLN-X}.} Suppose that \eqref{eqn:LLN-X} does not
hold and that on the contrary, for some $ \alpha \in (0, 1) $ and along a subsequence $ \{ t_{k} \}_{k \in \NN} $ such that $
t_{k} \uparrow \infty $, we have $ X_{t_{k}}< \gamma_{2} \alpha t_{k} $. From
the definition of $ X_{t} $ we deduce that
\begin{equation}\label{eqn:contraddiction}
    \frac{1}{\lceil \gamma_{2} \alpha t_{k} \rceil} \sum_{i =1}^{ \lceil
    \gamma_{2} \alpha t_{k} \rceil}
    S_{i} \geqslant \frac{t_{k}}{\lceil \gamma_{2} \alpha t_{k}\rceil} \sim
    \frac{1}{\gamma_{ 2 } \alpha }   \;, \qquad \forall k \in \NN \;.
\end{equation}

%Now if the $ S_{i} $ were i.i.d., by the strong law of large numbers we would have
%that
%\begin{align*}
    %\lim_{k \to \infty} \frac{1}{ \lceil \gamma_{2} \alpha t_{k} \rceil}
    %\sum_{i =1}^{ \lceil \gamma_{2} \alpha t_{k} \rceil}
    %S_{i} = \EE [S_{1}] \geqslant \frac{1}{\gamma_{2} \alpha} \;, 
%\end{align*}
%where the last inequality holds by \eqref{eqn:contraddiction} and is in
%contraddiction with the definition of $ \gamma_{2} $.

%Now, in our setting we can't use the law of large numbers directly. Instead,
Now let $ \mu_{i} = \EE[S_{i} \vert \widetilde{\mG}_{i}] $ where $
\widetilde{\mG}_{i} = \mF_{\tau_{i}} $ is the filtration generated by the excursion start times $
\{ \sigma_{i+1} \}_{i \in \NN} $. Then 
\begin{align*}
    M_{n} \stackrel{\mathrm{def}}{=} \sum_{i = 1}^{n} (S_{i} - \mu_{i} ) 
\end{align*}
is a martingale with respect to the discrete filtration $ (
\widetilde{\mG}_{i})_{i \in \NN} $. Hence, as above, by the moment bound
Lemma~\ref{lem:moment-bound}, we can conclde by the MLLN that
almost surely and in $ L^{1} $
\begin{align*}
    \frac{1}{n} M_{n} \to 0 \;.
\end{align*}
In particular, by our assumption on $ \alpha $ and $
t_{k} $ 
\begin{align*}
    \frac{1}{\alpha \gamma_{2}}
    \leqslant \liminf_{k \to \infty} \frac{1}{\gamma_{2} \alpha t_{k}} \sum_{i =
    1}^{\gamma_{2} \alpha t_{k}} S_{i} \leqslant \limsup_{k \to \infty} \frac{1}{\gamma_{2} \alpha
    t_{k}} \sum_{i =1}^{\gamma_{2} \alpha t_{k}} \mu_{i} \leqslant
    \frac{1}{\gamma_{2}} \;,
\end{align*}
which is a contraddiction, since $ \alpha \in (0, 1) $. This concludes the
proof of our result.
\end{proof}

We conclude with a moment estimate on our excursion times.

\begin{lemma}\label{lem:moment-bound}
    Under Assumption~\ref{assu:noise}, in the case $ \mf{a} < \infty $, there exist $ p(\mf{a}, d) \in [a, \infty) $,  $ 0< R_{1} < R_{2} < \infty $ and $
    0 < \overline{R}_{1} < \overline{R}_{2} < R_{3} < \infty  $ such that for some $
    \kappa > 0 $ and any $ B > 0 $:
    \begin{align*}
        \inf_{\mathbf{u} \in \tC^{\mathrm{m}}} \EE [ c(T^{\mathrm{out}}(\mathbf{u}) -
        T^{\mathrm{in}} (\mathbf{u}), B)]> 0\;, \quad
        \EE[ \exp \{ \kappa S(\mathbf{u}) \} ] < C(1 + \| \mathbf{u}
        \|_{L^{p}}^{p}) \;, \quad \forall \mathbf{u} \in L^{p} \;,
    \end{align*}
    where the center sets are defined in terms of $ R_{i},
    \overline{R}_{i} $ as in \eqref{eqn:center-Lp} and
    \eqref{eqn:center-Linf}.
\end{lemma}

\begin{proof}

    \textit{Bound on $ S (\mathbf{u}) $.} For $ \mathbf{u} \in L^{p} $,
    following the calculations in the proof of Proposition~\ref{prop:wp}, we have
    that for some $ c_{1}, c_{2} > 0 $:
    \begin{equation}\label{eqn:lyap}
        \ud \| \mathbf{u}_{t} \|_{L^{p}}^{p} \leqslant (- c_{1} \| \mathbf{u}
        \|_{L^{p}}^{p} + c_{2} ) \ud t + \ud \mathbf{M}^{(p)}_{t} \;,
    \end{equation}
    for a continuous square integrable martingale $
    \mathbf{M}^{(p)}_{t} $. Since for $ \mathbf{u} \in \tC^{\mathrm{m}} $ the
    bound is trivial, let us assume that $ \mathbf{u} \not\in
    \tC^{\mathrm{m}} $. Hence, for any $ \kappa > 0 $
    \begin{align*}
        \ud ( e^{\kappa t} \| \mathbf{u}_{t} \|_{L^{p}}^{p}) \leqslant
        e^{\kappa t} (- c_{1} \| \mathbf{u}
        \|_{L^{p}}^{p} + c_{2} + \kappa \| \mathbf{u} \|_{L^{p}}^{p} ) \ud t +
        e^{\kappa t}\ud \mathbf{M}^{(p)}_{t}\;.
    \end{align*}
    Assuming that $ R_{1}^{p} \geqslant  2 c_{2} c_{1}^{-1} +1$  and $ \kappa \leqslant
    c_{1} /2$, we deduce that for any $ n \in \NN $
    \begin{align*}
        e^{\kappa S \wedge n} \| \mathbf{u}_{S \wedge n}
        \|_{L^{p}}^{p}  + \frac{c_{1}}{2}
        \int_{0}^{S \wedge n} e^{\kappa r} \ud r \leqslant
        \int_{0}^{S \wedge n} e^{\kappa r} \ud \mathbf{M}^{(p)}_{r} + \|
        \mathbf{u}_{0} \|_{L^{p}}^{p} \;.
    \end{align*}
    In particular we conclude that
    \begin{align*}
        \kappa^{-1} \EE \left[ e^{\kappa S \wedge n} -1 \right] = \EE \left[ \int_{0}^{S
        \wedge n} e^{\kappa r}\ud r \right] \leqslant
        2 c_{1}^{-1} \| \mathbf{u} \|^{p} \;.
    \end{align*}
    Sending $ n \to \infty $ allows us to deduce that $
    \EE [\exp \{ \kappa S \}] < C (\kappa) (1 + \| \mathbf{u}
    \|_{L^{p}}^{p}) $, which is the required bound.
%\begin{align*}
    %\PP_{\mathbf{u}} (S(\mathbf{u}) \geqslant s ) 
    %& \leqslant \frac{\EE_{\mathbf{u}} [\| u_{s} \|^{2} 1_{\{ \sup_{0 \leqslant
                    %r \leqslant s}  \| u_{r} \| \geqslant c^{+} \}} ]
    %}{(c^{+})^{2}} \leqslant \frac{e^{- \lambda s} (c^{-})^{2}}{(c^{+})^{2}}
    %\;.
%\end{align*}
%\tommasoText{But is the last estimate true?}

    \textit{Bound on $ T(\mathbf{u}) $.} It suffices to prove that  
    \begin{equation}\label{eqn:aim-mmt-1}
        \begin{aligned}
            \inf_{\mathbf{u} \in \tC^{\mathrm{m}}} \PP (T^{\mathrm{in}}
            (\mathbf{u}) < T^{\mathrm{out}}(\mathbf{u}) ) > 0 \;.
        \end{aligned}
    \end{equation}
    As a first step towards deducing this bound, we will prove that for any $
    \ve, \delta \in (0, 1) $, provided $ 0 < R_{1}(\ve, \delta) <
    R_{2}(\ve, \delta)< R_{3}(\ve, \delta) $ and $ 0 < \overline{R}_{1}(\delta,
    \ve)< \overline{R}_{2} (\delta, \ve) < \infty$ are sufficiently large, we
    have that
    \begin{equation}\label{eqn:aim-mmt-2}
        \begin{aligned}
            \inf_{\mathbf{u} \in \tC^{\mathrm{m}}}  \PP ( T (\mathbf{u}) > 2 \delta) &>
            1 - \ve \;, \\
            \inf_{\mathbf{u} \in \tC^{\mathrm{m}}}  \PP ( T^{\mathrm{in}}(\mathbf{u})
            < \delta) &> 1 - \ve \;.
        \end{aligned}
    \end{equation}
%    To obtain these estimates we start with the following supplementary bound, which
    %holds for any $ \delta \in (0, 1) $ and sufficiently large $  R_{3} (\ve ,
    %\delta) \in (0, \infty) $:
    %\begin{align*}
        %\inf_{\mathbf{u} \in \tC^{\mathrm{m}}} \PP ( T (\mathbf{u}) > \delta) >
        %1 - \ve \;,
    %\end{align*}
    For the first bound we use \eqref{eqn:lyap} to find that for $
    \mathbf{u}_{0} \in \tC^{\mathrm{m}} $ and a square integrable martingale $
    \mathbf{M}^{(p)}_{t} $ we have
    \begin{align*}
        \| \mathbf{u}_{t} \|_{L^{p}}^{p} \leqslant R_{2}^{p} + c_{2} t +
        \mathbf{M}^{(p)}_{t} \;.
    \end{align*}
    Hence 
    \begin{align*}
        \PP( T (\mathbf{u}) \leqslant  2 \delta) \leqslant \PP \left(
            R^{p}_{2} + c_{2} 2 \delta + \sup_{0 \leqslant s \leqslant 2 \delta}
        \mathbf{M}^{(p)}_{s} \geqslant R_{3}^{p} \right) = \PP  \left( \sup_{0
        \leqslant s \leqslant 2 \delta} \mathbf{M}^{(p)}_{s} \geqslant R_{3}^{p}
        - R^{p}_{2} - c_{2} 2 \delta \right)\;.
    \end{align*}
    Now if we define $ \lambda = R^{p}_{3} - R^{p}_{2} -2 c_{2} $ (since $
    \delta \in (0, 1)$), and assuming $ \lambda > 0 $ by choosing $
    R_{3} $ sufficiently large, we obtain, by using the definition of the
    martingales in \eqref{eqn:martingale}
    \begin{align*}
        \PP( T (\mathbf{u}) \leqslant  2 \delta) & \leqslant \PP( \sup_{0
        \leqslant s \leqslant 2 \delta} \mathbf{M}^{(p)}_{s} \geqslant \lambda
    )\\
    & \lesssim \lambda^{-2} \EE[ |\mathbf{M}^{(p)}_{2 \delta \wedge
    T(\mathbf{u})}|^{2}] \simeq \EE
    \langle \mathbf{M}^{(p)} \rangle_{ 2 \delta \wedge T (\mathbf{u})} \lesssim_{\psi}
    \lambda^{-2} \EE \int_{0}^{2 \delta \wedge T(\mathbf{u})} \|
    \mathbf{u}_{s} \|_{L^{p-1}}^{2(p-1)}  \ud s \\
    & \lesssim \lambda^{- 2} R_{3}^{2 (p-1)}  \delta \sim R^{-2}_{3} \delta\;,
    \end{align*}
    from the definition of $ \lambda $. Therefore choosing $ R_{3} ( \ve,
    \delta) \gg 1 $ so that $ R_{3}^{- 2} \delta \lesssim \ve $ we obtain the
    desired estimate. Now we pass to a bound on $
    T^{\mathrm{in}}(\mathbf{u}) $. Following \eqref{eqn:final-l-inf}, for $
    p \geqslant \mf{a} $ sufficiently large so that
    \begin{align*}
        \gamma \stackrel{\mathrm{def}}{=} 1 + \frac{d \mf{ a} }{p} + 2 \kappa  < 2 \;,
    \end{align*}
    we find that
        \begin{align*}
            \| \mathbf{u}_{\delta} \|_{\infty} \lesssim
            \delta^{ -\frac{\gamma}{2}}  \left\{ \| 
            \mathbf{u} \|_{L^{p/ \mf{a}}} + \delta \sup_{0 \leqslant s \leqslant \delta
            } \| \mathbf{u}_{s} \|_{L^{p}}^{\mf{a}} + \sup_{0
        \leqslant s \leqslant \delta} \| z_{ s} \|_{\infty} \right\} \;.
        \end{align*}
    Therefore we obtain
    \begin{align*}
        \PP( T^{\mathrm{in}} (\mathbf{u}) \geqslant \delta) & \leqslant \PP( \sup_{0
        \leqslant s \leqslant \delta } \| \mathbf{u}_{s} \|_{\infty} \geqslant
        \overline{R}_{1} \;, T(\mathbf{u}) \geqslant  \delta) \leqslant \PP
        ( \| \mathbf{u}_{\delta} \|_{\infty} \geqslant \overline{R}_{1} \;,
        T(\mathbf{u}) \geqslant \delta ) \\
        & \leqslant \overline{R}_{1}^{-1} \EE \big[ \| \mathbf{u}_{\delta}
        \|_{\infty} 1_{\{T(\mathbf{u}) \geqslant \delta \}} \big]  \lesssim
        \overline{R}_{1}^{-1} \delta^{- \frac{\gamma}{2}} \left\{ 
        R_{3}^{\mf{a}} + \EE \left[ \sup_{0 \leqslant s \leqslant \delta} \|
        z_{s} \|_{\infty}\right]  \right\} \;.
    \end{align*}
    Now, choosing $ \overline{R}_{1} (\delta, \ve) \gg 1 $ sufficiently large we obtain the desired estimate
    \begin{align*}
        \sup_{\mathbf{u} \in \tC^{\mathrm{m}}} \PP (T^{\mathrm{in}}
        (\mathbf{u}) \geqslant \delta  ) < \ve \;,
    \end{align*}
    which concludes the proof of \eqref{eqn:aim-mmt-2}. We now pass to
    \eqref{eqn:aim-mmt-1}, where we have
    \begin{align*}
        \PP ( T^{\mathrm{out}}(\mathbf{u}) - T^{\mathrm{in}}
        (\mathbf{u}) \leqslant \delta) & \leqslant \PP(
        T^{\mathrm{out}}(\mathbf{u}) - T^{\mathrm{in}}(\mathbf{u}) \leqslant
        \delta \;, T (\mathbf{u}) \geqslant 2 \delta \;,
        T^{\mathrm{in}}(\mathbf{u}) <  \delta) + 2 \ve \\
        & \leqslant \overline{R}_{2}^{- 1} \EE \left[
        \sup_{T^{\mathrm{in}}(\mathbf{u}) \leqslant s \leqslant 2 \delta} \|
        u_{s} \|_{\infty} 1_{\{ T (\mathbf{u}) \geqslant 2 \delta ,
    T^{\mathrm{in}}(\mathbf{u}) < \delta \}} \right] + 2 \ve \;.
    \end{align*}
    Now, to estimate the average above we follow again the bound
    \eqref{eqn:final-l-inf}, this time with initial condition in $
    L^{\infty} $. We obtain, for $ p (\mf{a}, d) \geqslant \mf{a} $
    sufficiently large
    \begin{equation}\label{eqn:mmt-conclusion}
        \sup_{T^{\mathrm{in}}(\mathbf{u}) \leqslant s \leqslant \delta}\| u_{s}
        \|_{\infty} \lesssim \| \mathbf{u}_{T^{\mathrm{in}}} \|_{\infty} + 2
        \delta \sup_{T^{\mathrm{in}}(\mathbf{u}) \leqslant s \leqslant \delta} \| \mathbf{u}_{s}
        \|_{L^{p}}^{\mf{a}} + \sup_{T^{\mathrm{in}}(\mathbf{u}) \leqslant s
        \leqslant 2 \delta} \| z_{T^{\mathrm{in}}(\mathbf{u}), t}
        \|_{\infty} \;.
    \end{equation}
    We conclude by Proposition~\ref{prop:wp} that
    \begin{align*}
        \EE \left[ \sup_{T^{\mathrm{in}}(\mathbf{u}) \leqslant s \leqslant 2 \delta} \|
        u_{s} \|_{\infty} 1_{\{ T (\mathbf{u}) \geqslant 2 \delta ,
    T^{\mathrm{in}}(\mathbf{u}) < \delta \}} \right] \lesssim
    \overline{R}_{1} + C \delta (1 + R_{3}^{p}) + \EE \left[\sup_{T^{\mathrm{in}}(\mathbf{u}) \leqslant s
        \leqslant 2 \delta} \| z_{T^{\mathrm{in}}(\mathbf{u}), t}
    \|_{\infty}  \right]\;,
    \end{align*}
    so that by the bounds on $ z $ in Lemma~\ref{lem:reg-z} and by choosing $
    \overline{R}_{2} \gg 1 $ sufficiently large, the bound
    \eqref{eqn:aim-mmt-1} follows from \eqref{eqn:mmt-conclusion}. Hence the
    proof of our result is concluded.

\end{proof}

\newcommand{\etalchar}[1]{$^{#1}$}
\def\polhk#1{\setbox0=\hbox{#1}{\ooalign{\hidewidth
  \lower1.5ex\hbox{`}\hidewidth\crcr\unhbox0}}}


\begin{thebibliography}{DDG{\etalchar{+}}22}

\bibitem[Bor14]{Boritchev2014Turbulence}
A.~Boritchev.
\newblock Turbulence in the generalised {B}urgers equation.
\newblock {\em Uspekhi Mat. Nauk}, 69(6(420)):3--44, 2014.

\bibitem[Bor16]{Boritchev2016MultidimBurgers}
A.~Boritchev.
\newblock Multidimensional potential {B}urgers turbulence.
\newblock {\em Comm. Math. Phys.}, 342(2):441--489, 2016.

\bibitem[Bor18]{Boritchev18Exponential}
A.~Boritchev.
\newblock Exponential convergence to the stationary measure for a class of 1{D}
  {L}agrangian systems with random forcing.
\newblock {\em Stoch. Partial Differ. Equ. Anal. Comput.}, 6(1):109--123, 2018.

\bibitem[DDG{\etalchar{+}}22]{Dunlap22sclLine}
T.~D. Drivas, A.~Dunlap, C.~Graham, J.~La, and L.~Ryzhik.
\newblock Invariant measures for stochastic conservation laws on the line,
  2022.

\bibitem[DGR21]{Dunlap21BurgersLine}
A.~Dunlap, C.~Graham, and L.~Ryzhik.
\newblock Stationary solutions to the stochastic {B}urgers equation on the
  line.
\newblock {\em Comm. Math. Phys.}, 382(2):875--949, 2021.

\bibitem[DjS22]{djurdjevac2022stabilisation}
A.~Djurdjevac and A.~Shirikyan.
\newblock Stabilisation of a viscous conservation law by a one-dimensional
  external force.
\newblock {\em arXiv preprint arXiv:2204.03427}, 2022.

\bibitem[DV10]{DebusscheVovelle2012Existence}
A.~Debussche and J.~Vovelle.
\newblock Scalar conservation laws with stochastic forcing.
\newblock {\em J. Funct. Anal.}, 259(4):1014--1042, 2010.

\bibitem[DV15]{DebusscheVovelle2015Invariant}
A.~Debussche and J.~Vovelle.
\newblock Invariant measure of scalar first-order conservation laws with
  stochastic forcing.
\newblock {\em Probab. Theory Related Fields}, 163(3-4):575--611, 2015.

\bibitem[EKMS00]{WeinanKhaninMazelSinai}
Weinan E, K.~Khanin, A.~Mazel, and Ya. Sinai.
\newblock Invariant measures for {B}urgers equation with stochastic forcing.
\newblock {\em Ann. of Math. (2)}, 151(3):877--960, 2000.

\bibitem[GP17]{GubinelliPerkowski17KPZReloaded}
M.~Gubinelli and N.~Perkowski.
\newblock K{PZ} reloaded.
\newblock {\em Comm. Math. Phys.}, 349(1):165--269, 2017.

\bibitem[GR00]{Gyiongy00Wellposed}
I.~Gy\"{o}ngy and C.~Rovira.
\newblock On {$L^p$}-solutions of semilinear stochastic partial differential
  equations.
\newblock {\em Stochastic Process. Appl.}, 90(1):83--108, 2000.

\bibitem[Hai13]{Hairer13Solving}
M.~Hairer.
\newblock Solving the {KPZ} equation.
\newblock {\em Ann. of Math. (2)}, 178(2):559--664, 2013.

\bibitem[Lun95]{Lunardi95Analytic}
A.~Lunardi.
\newblock {\em Analytic semigroups and optimal regularity in parabolic
  problems}.
\newblock Modern Birkh\"{a}user Classics. Birkh\"{a}user/Springer Basel AG,
  Basel, 1995.
\newblock [2013 reprint of the 1995 original] [MR1329547].

\bibitem[MW17]{MourratWeber17Infinity}
J.-C. Mourrat and H.~Weber.
\newblock The dynamic {$\Phi^4_3$} model comes down from infinity.
\newblock {\em Comm. Math. Phys.}, 356(3):673--753, 2017.

\bibitem[PR19]{PerkowskiRosati2019KPZ}
N.~Perkowski and T.~Rosati.
\newblock The {KPZ} equation on the real line.
\newblock {\em Electron. J. Probab.}, 24:Paper No. 117, 56, 2019.

\bibitem[Ros21]{Rosati21Lyap}
T.~Rosati.
\newblock Lyapunov exponents in a slow environment, 2021.

\bibitem[Ros22]{Rosati22Synchro}
T.~Rosati.
\newblock Synchronization for {KPZ}.
\newblock {\em Stoch. Dyn.}, 22(4):Paper No. 2250010, 46, 2022.

\bibitem[Shi18a]{Shirikyan18ControlTheory}
A.~Shirikyan.
\newblock Control theory for the {B}urgers equation: {A}grachev-{S}arychev
  approach.
\newblock {\em Pure Appl. Funct. Anal.}, 3(1):219--240, 2018.

\bibitem[Shi18b]{Shirikyan18Mixing}
A.~Shirikyan.
\newblock Mixing for the {B}urgers equation driven by a localized
  two-dimensional stochastic forcing.
\newblock In {\em Evolution equations: long time behavior and control}, volume
  439 of {\em London Math. Soc. Lecture Note Ser.}, pages 179--194. Cambridge
  Univ. Press, Cambridge, 2018.

\bibitem[Sin91]{Sinai1991Buergers}
Y.~G. Sinai.
\newblock Two results concerning asymptotic behavior of solutions of the
  {B}urgers equation with force.
\newblock {\em J. Statist. Phys.}, 64(1-2):1--12, 1991.

\bibitem[ZZZ22]{ZhuZhu2022HJB}
X.~Zhang, R.~Zhu, and X.~Zhu.
\newblock Singular {HJB} equations with applications to {KPZ} on the real line.
\newblock {\em Probab. Theory Related Fields}, 183(3-4):789--869, 2022.

\end{thebibliography}
\end{document}